\documentclass[10pt,fleqn,a4paper]{article}

\usepackage{amsmath}
\usepackage{amssymb}
\usepackage{mathrsfs}
\usepackage{cancel}
\usepackage{authblk}
\usepackage{amsthm}
\usepackage{listings}

\usepackage[dvipsnames]{xcolor}
\usepackage{graphicx}

\usepackage{url}
\usepackage{hyperref}

\newtheorem{defn}{Definition}
\newtheorem{thm}{Theorem}

\providecommand{\keywords}[1]
{
  \small	
  \textbf{\textit{Keywords---}} #1
}

\begin{document}

\title{Existence of the C-type renormalisation two-cycle}
\date{}

\author[1]{Zainab Rahman}
\author[2]{Maria Pickett}
\author[3]{Andrew Burbanks\thanks{Corresponding Author: andrew.burbanks@port.ac.uk}}
\affil[1,2,3]{School of Mathematics and Physics, University of Portsmouth, Portsmouth, United Kingdom}

\maketitle

\begin{abstract}
We prove the existence of the C-type renormalisation two-cycle, helping to establish the universality of the C-type route to chaos in families of non-invertible maps of the plane.

Families of two-dimensional non-invertible maps with at least two parameters and critical points arising from a fold singularity
(one of the fundamental types, alongside the cusp and the projection of the Whitney umbrella),
exhibit a distinct type of critical scaling: the C-type. An accumulation of parameter values leads to an infinite collection of coexisting attracting cycles of periods $4^n$ or $2\cdot 4^n$ for which, asymptotically, quadrupling the period is accompanied by parameter-space and state-space scaling governed by particular universal constants.
Kuznetsov et al. provided an explanation for this phenomenon in terms of a stationary orbit of period two of the Renormalisation Group (RG) transformation for period-doubling.

We prove the existence of the renormalisation two-cycle
in a Banach space of analytic maps and gain rigorous bounds on the corresponding universal state space scaling constants. This result provides a further step in proving a series of outstanding conjectures concerning distinct universality classes for period-doubling. It extends recent results for unidirectionally-coupled maps (the FS-type) to bidirectionally-coupled maps, and generalises the framework from fixed points to periodic orbits of the corresponding renormalisation operators. It provides a further step in establishing the conjectured picture that the C-type universality class is born from the FS-type class via a period-doubling bifurcation in the dynamics of the RG transformation itself.
The proof relies on rigorous (validated)
computations to establish that a variant of Newton's method for the two-cycle is a contraction map.

The C-type scaling regularity occurs in a number of dynamical systems of interest, perhaps most notably in biologically-plausible models of nephron blood pressure autoregulation.
\end{abstract}


\textit{2020 Mathematics Subject Classification (MSC).}
Primary 37E20; Secondary 37E99

\section{Introduction}

\paragraph{Renormalisation and Universality}

Low-dimensional discrete-time dynamical systems play a pivotal role in understanding behaviour in higher-dimensional systems of differential equations, where differential expansion and contraction (and dissipation) often lead to effective dynamics on lower-dimensional manifolds and surfaces of section.
A key step in understanding transitions to chaos in low-dimensional discrete-time systems came with  Feigenbaum's observations~\cite{feigenbaum_quantitative_1978,feigenbaum_universal_1979} alongside Coullet and Tresser~\cite{coullet_iterations_1978} that period-doubling cascades in non-invertible maps of the interval share universal properties, falling into distinct universality classes determined by simple features of the map. For one-parameter families of unimodal maps on the interval with quadratic critical points, exemplified by the model family $f_A:x\mapsto 1-Ax^2$, infinite period-doubling cascades occur. In suitable coordinates, asymptotic scaling regularity is observed in the limit of high period: doubling of the period is accompanied by parameter scaling (in $A$) by a universal constant $1/\delta$ with $\delta\approx 4.669$ and state space scaling (in $x$) by another universal constant $1/\alpha$ where $\alpha\approx-2.502$.

An explanation for this regularity was provided in terms of an RG transformation that encodes the observed scaling
behaviour by composing the map with itself (halving the period of doubled orbits), expressed by the Cvitanovi\'c-Feigenbaum~operator,
\begin{equation*}
R(g)(x):= a^{-1} g(g(ax)),
\end{equation*}
where $a := g(g(0))$, chosen to enforce the normalisation $R(g)(0)=1$ (in which coordinates have been chosen to put the critical point at the origin).
The explanation for universality in period-doubling rests on the existence of a nontrivial stationary point $g^*$ of the transformation, for which $R(g^*)=g^*$, and its hyperbolicity. The stationary point is conjectured to be of saddle type: restricting attention to the dynamically-relevant eigenvalues (i.e., those that do not correspond to infinitesimal coordinate changes), the spectrum of the derivative $\mathfrak{D}R(g^*)$ at the fixed point has a single expanding eigenvalue $\delta$, with the rest of the spectrum lying inside the open unit disc. Correspondingly, $g^*$ has a one-dimensional unstable manifold that determines a universal family of functions undergoing a period-doubling cascade, and a codimension one stable manifold. This function-space picture provides an explanation for the observed universality that extends to a broad class of  maps having a quadratic critical point, including those derived from higher-dimensional systems of differential equations. Similar universal-scaling scenarios have been identified in a wide variety of low-dimensional dynamical systems. Cvitanovi\'c~\cite{cvitanovic_universality_2017} provides a collection of the establishing literature in this area and Khanin~\cite{khanin_renormalization_2019} summarises the development of the application of the RG approach.

\paragraph{Critical scaling for non-invertible maps of the plane.}
Here, we consider families of two-dimensional non-invertible maps that display a period-doubling cascade. Scaling behaviour is observed at the accumulation of period doublings that appears to fall into universality classes determined by the nature of a critical point of the map.
Three generic types of critical points for two-dimensional non-invertible maps were identified in Kuznetsov et al. ~\cite{kuznetsov_perioddoubling_1997} corresponding to folds, cusps, and to a projection of the Whitney umbrella. The fold and Whitney umbrella were found to correspond to distinct universality classes for period-doubling, known as the C-type (`Cycle') and FQ-type (`Feigenbaum-Quasiperiodicity'), respectively. In a forthcoming publication, we examine the FQ-type scaling, that occurs in systems where period-doubling and Neimark-Sacker bifurcations coincide, proving the existence of a stationary point of the corresponding renormalisation group transformation. In this paper, we examine the C-type scaling, observed for certain classes of systems where period-doublings interact with saddle-node bifurcations.

Specifically, Kuznetsov et al.~\cite{kuznetsov_perioddoubling_1997,kuznetsov_multiparameter_2005,kuznetsov_birth_2008} found critical scaling behaviour associated with maps of fold type in families with at least two parameters.
A prototypical example is given by the following three-parameter family,
\begin{align}
    x_{k+1} &= g(x_k,y_k) := A - x_k^2 + B y_k,
    \qquad
    y_{k+1} = f(x_k,y_k) := -x_k^2 + C y_k,\label{eqn:model12}
\end{align}
obtained by composing a fold map with a general affine transformation of the plane, followed by a change of coordinates to reduce the number of free parameters.

For the above system, fixing the value of the parameter $B$ gives rise to a sequence of period-doubling bifurcation curves in the $(A,C)$-plane. In suitable (scaling) coordinates, following the locations of the terminal points of these curves, for successive period doublings, leads to two complementary sets of asymptotic scaling symmetries. Depending on the parameter $B$, there is either a set of coexisting attracting cycles of periods $4^n$, for $n=0,1,2,\ldots$ with an accompanying set of unstable cycles of periods $2\cdot 4^n$ or the converse (with unstable periods $4^n$ and stable periods $2\cdot 4^n$). For both collections of cycles, in the limit of high periods, quadrupling of the period is accompanied by parameter space scaling by universal constants $1/\delta_1$ and $1/\delta_2$ (where $\delta_1\approx 92.4$ and $\delta_2\approx 4.19$) and state space scaling by universal constants $a^*=1/\alpha^*$ and $b^*=1/\beta^*$ (where $\alpha^*\approx 6.57$ and $\beta^*\approx 22.1$). The resulting set of coexisting attracting periodic orbits is sometimes termed a \emph{critical quasi-attractor}.

Although the phenomenon is, strictly speaking, one of codimension two, the authors note that the spectrum of the derivative of the appropriate operator has a contracting relevant (i.e., non-coordinate-change) eigenvalue $\lambda\approx 0.93$ relatively close to $1$ which may hinder the observation of the scaling scenario (at least for relatively low $n$) unless care is taken to tune any available parameters to suppress the corresponding mode (in the model three-parameter system given above in~(\ref{eqn:model12}), the parameter $B$ may be tuned in order to achieve this).

Kuznetsov et al.~\cite{kuznetsov_perioddoubling_1997} performed a thorough numerical investigation and proposed an explanation for the universal scaling in terms of a stationary orbit (here, a two-cycle) of the corresponding renormalisation group transformation that acts on doubled orbits by halving their periods and then rescaling.
In this paper, we prove that the conjectured RG two-cycle exists.

\paragraph{Computer-assisted proofs for RG transformations.}
Computer-assisted proofs for RG transformations, and their associated universal scaling scenarios, go back to
Lanford~\cite{lanford_computerassisted_1984}, who provided the first computer-assisted proof of existence and hyperbolicity of the RG fixed point for period-doubling universality in families of unimodal maps of the interval with quadratic critical points, thereby giving the first proof of some of the Feigenbaum conjectures~\cite{feigenbaum_quantitative_1978,feigenbaum_universal_1979}.
Eckmann and Wittwer~\cite{eckmann_computer_1985} extended the work to families of interval maps with critical points of integer degree $d=2m$ for $m\to\infty$ by making use of Borel summability and Ecalle's theory of resurgent functions.
Eckmann, Koch, and Wittwer~\cite{eckmann_computer-assisted_1984} provided a computer-assisted proof of period-doubling universality for area-preserving maps.
Mestel~\cite{mestel_computer_1985} established existence of the relevant RG fixed point for circle maps, and
Stirnemann~\cite{stirnemann_renormalization_1993,stirnemann_existence_1994} for pairs of maps in one complex variable, relevant to universality of critical scaling at the boundary of Siegel discs studied by Burbanks, Osbaldestin, and Stirnemann~\cite{burbanks_holder_1995,burbanks_renormalization_1997,burbanks_rigorous_1998,burbanks_fractal_1998}, in the case of golden mean rotation number. Notable works include Koch and Wittwer~\cite{koch_nongaussian_1986,koch_nontrivial_1994}, 
Gaidashev and Koch~\cite{gaidashev_period_2011a},
Koch~\cite{koch_renormalization_2004,koch_2023},
Arioli and Koch~\cite{arioli_critical_2010} who proved the existence of MacKay's RG fixed point for area-preserving maps~\cite{mackay_renormalisation_1993}, and Gaidashev and Yampolsky~\cite{gaidashev_golden_2016} who completed the picture for golden mean Siegel disc universality.
More recently,
Burbanks, Osbaldestin, and Thurlby~\cite{burbanks_rigorous_2021a} gave computer-assisted proofs for RG fixed points and related eigenfunctions in one-variable systems and unidirectionally-coupled pairs of nonivertible maps of two variables (FS-type)~\cite{burbanks_existence_2023}
in which it was possible to make a reduction to spaces of one-variable maps.
The FS-type scaling is relevant to coupled systems in which one subsystem undergoes the usual period-doubling route to chaos, and the  second accumulates an integral characteristic of the dynamics of the first.


\paragraph{Contribution.} This paper extends the recent work on FS-type criticality for noninvertible maps by moving to systems of fully-coupled pairs of maps of two variables.
The work also proves an open conjecture concerning critical scaling scenarios that may be generic for non-invertible maps of the plane (specifically those corresponding to fold singularities and to the projection of the Whitney umbrella). Together with a forthcoming publication, in which we address existence of the RG fixed point for bidirectionally coupled systems of FQ-type, this helps to complete the characterisation of scaling scenarios for non-invertible maps in~\cite{kuznetsov_perioddoubling_1997}.
This paper also extends the framework from finding fixed points of the RG transformations for coupled pairs of maps of two variables, to proving the existence of periodic cycles. We explain how to find good domains for the constituent maps of cycles and how to formulate (and bound) an appropriate auxiliary operator.
Finally, it was conjectured~\cite{kuznetsov_birth_2008} that the family of C-type two-cycles (for continuously-varying degree $d$ at the critical point) combines with the family of FS-type RG fixed points at a period-doubling bifurcation in the dynamics of the RG operator itself (an epsilon expansion for the two-cycle, perturbing the degree $d$ of the critical point, which is treated analogously to a spatial dimension, is computed in~\cite{kuznetsov_birth_2008} based on this idea). This work provides a further step in establishing the conjectured picture rigorously.
Applications of the C-type universality to dynamical systems of interest include driven R\"ossler-type  oscillators at the edges of synchronisation tongues~\cite{kuznetsov_universality_2001} and coupled electrical circuits. A notable example from Biology occurs in the study of nephron blood-pressure autoregulation: nephrons are functional units in the kidney that respond to changes in arterial blood pressure to help ensure a stable filtration rate. The dynamics of these units is somewhat under-damped and as a result they appear to display a rich collection of dynamical behaviours including sustained oscillations, mode-locking, and period-doubling bifurcations. Biologically-inspired models have been shown to display the C-type criticality~\cite{laugesen_modelling_2011,laugesen_ctype_2011}.

\subsection{Overview}

In Section~\ref{sec:rgops}, we define the relevant renormalisation operator, $R$, and its formal Fr\'echet derivative. The underlying symmetry of families of maps with a critical point of fold-type enables us to take an ansatz for the maps of the desired two-cycle, leading to an equivalent induced operator, $T$ and its formal Fr\'echet derivative.
In Section~\ref{sec:fnspaces}, we describe the basic (analytic) function spaces in which we will work, first defined with respect to the unit polydisc and then defined with respect to general polydiscs.
In Section~\ref{sec:cycles}, we reformulate the problem of finding period-two function-pairs for the operator $T$ in terms of a root-finding problem for the two-cycle on the corresponding product space and, in Section~\ref{sec:approxcycle1}, we use an approximate Newton method on polynomials to find an initial approximate two-cycle of $T$. In order to move to the setting of analytic maps (and the corresponding function spaces), we then need to find suitable domains on which the operator $T$ is well-defined, in Section~\ref{sec:de}.
In order to proceed further, we need to place additional constraints on the domains of the maps (and hence on the corresponding functions spaces) in order that suitable bounds can be maintained during the computer-assisted portion of the proof. To justify this, we first summarise briefly the mathematical framework used for the computer-assisted proof and provide sufficient details to show how bounds are computed and maintained during the composition of maps, in Section~\ref{sec:rigorousframework}.
We discuss the implications for the choice of domains in Section~\ref{sec:domaindetails}.
In Section~\ref{sec:quasinewton}, we formulate a quasi-Newton operator and indicate how we establish
that it is a uniform contraction map on a carefully chosen ball around the approximate two-cycle.
Details of the computer-assisted portion of the proof and the computed bounds 
are given in Section~\ref{sec:results}. We briefly discuss computational issues in Section~\ref{sec:computational}. Conclusions and suggestions for future work are provided in Section~\ref{sec:conclusions}. To ensure that the results may be reproduced, the code for the computer-assisted portion of the proof is given in a Zenodo repository~\cite{burbanks_2026_ctype_code}.

\section{Renormalisation group transformation}\label{sec:rgops}

In suitable (scaling) coordinates, the asymptotic self-similarity observed in period-doubling cascades of families of maps of two variables may be understood in terms of the following renormalisation group transformation (or renormalisation operator), $R$, that acts on period-doubled orbits by halving their periods and rescaling in the two variables independently.
\begin{defn}[Renormalisation Group transformation]\label{defn:renormop} For pairs of functions $g,f:\mathbb{R}^2\to\mathbb{R}$, we define the renormalisation (doubling) operator $R$ by
\begin{align*}
R(g, f)(x, y)
&= \left(\begin{array}{l}
a^{-1}g(g(ax, by), f(ax, by))\\
b^{-1}f(g(ax, by), f(ax, by))
\end{array}\right),
\end{align*}
where $a:=g(g(0,0),f(0,0))$ and $b:=f(g(0,0),f(0,0))$ which enforces the normalisation $R(g,f)(0,0)=(1,1)$.\end{defn}

In investigating different types of scaling behaviour (corresponding to distinct universality classes), we may work with equivalent operators that encode the symmetry of each type.

The asymptotic similarity between state space regions at successively smaller scales corresponding to quadrupling of the period in C-type systems, can be understood in terms of a two-cycle of $R$: we would like to find pairs of functions $(g_1,f_1)$ and $(g_2,f_2)$, defined on domains $\Omega_1,\Omega_2\subset\mathbb{C}^2$ (respectively), such that
\begin{align*}
R(g_1,f_1) &= (g_2,f_2),\qquad
R(g_2,f_2) = (g_1,f_1).
\end{align*}
Equivalently, we would like to solve the functional equations
\begin{align}
a_1^{-1}g_1(g_1(a_1x,b_1y),f_1(a_1x,b_1y))
&= g_2(x,y)\quad\mbox{on $\Omega_2$},\label{eqn:fnalr1}\\
a_1^{-1}f_1(g_1(a_1x,b_1y),f_1(a_1x,b_1y))
&= f_2(x,y)\quad\mbox{on $\Omega_2$},\label{eqn:fnalr2}\\
a_2^{-1}g_2(g_2(a_2x,b_2y),f_2(a_2x,b_2y))
&= g_1(x,y)\quad\mbox{on $\Omega_1$},\label{eqn:fnalr3}\\
a_2^{-1}f_2(g_2(a_2x,b_2y),f_2(a_2x,b_2y))
&= f_1(x,y)\quad\mbox{on $\Omega_1$},\label{eqn:fnalr4}
\end{align}
where $a_k:=g_k(g_k(0,0),f_k(0,0))$ and $b_k:=f_k(g_k(0,0),f_k(0,0))$ which enforces the normalisation $R(g_k,f_k)(0,0)=(1,1)$. These equations may be regarded as a two-dimensional generalisation of the Cvitanovi\'c-Feigenbaum equation~\cite{feigenbaum_quantitative_1978,feigenbaum_universal_1979,cvitanovic_universality_2017}. We denote the solution to the above functional equations by $(g_1^*,f_1^*,g_2^*,f_2^*)$ and the values for the corresponding scaling constants by $a_1^*,b_1^*,a_2^*,b_2^*$.

In discussing the stability of periodic orbits of the RG transformation (and, indeed, when proving their existence) we make use of the linearised operators $L_k:=\mathfrak{D}R(g_k,f_k)$.
Let the rescaled maps be denoted ${\tilde{g}}=g(ax,\,by)$ and ${\tilde{f}}=f(ax,\,by)$, then the formal Fr\'echet derivative of the operator $R$ at a particular pair $(g,f)$ acts on perturbations $(\delta g,\delta f)$ as follows,
\begin{align*}\label{eqn:two-var-frechet2}
&\mathfrak{D}R(g,\ f)({\delta g},\ {\delta f})(x,\ y)={}\nonumber\\
&\qquad\left(
\begin{array}{c}
-a^{-2}{\delta a}\cdot g(\tilde{g},\tilde{f})
+a^{-1}\bigl({\delta g}(\tilde{g}, \tilde{f})
    +\partial_1g(\tilde{g}, \tilde{f})\cdot{\delta\tilde{g}}
    +\partial_2g(\tilde{g}, \tilde{f})\cdot{\delta\tilde{f}}
\bigr)\\
-b^{-2}{\delta b}\cdot f(\tilde{g},\tilde{f})
+b^{-1}\bigl({\delta f}(\tilde{g}, \tilde{f})
    +\partial_1f(\tilde{g}, \tilde{f})\cdot{\delta\tilde{g}}
    +\partial_2f(\tilde{g}, \tilde{f})\cdot{\delta\tilde{f}}
\bigr)
\end{array}\right),
\end{align*}%
where the perturbations $\delta\tilde{g}$, $\delta\tilde{f}$ are given by
\begin{align*}
{\delta\tilde{g}}
&=
{\delta g}(ax,\,by)
+\partial_1g(ax,\,by)\cdot{\delta a}\cdot x
+\partial_2g(ax,\,by)\cdot{\delta b}\cdot y,\\
{\delta\tilde{f}}
&=
{\delta f}(ax,\,by)
+\partial_1f(ax,\,by)\cdot{\delta a}\cdot x
+\partial_2f(ax,\,by)\cdot{\delta b}\cdot y,
\end{align*}
and
\begin{align*}
{\delta a}
&=
{\delta g}(g(0,0),\ f(0,0))
+ \partial_1g(g(0,0),\ f(0,0))\cdot {\delta g}(0,0)\nonumber\\
&\quad{}+ \partial_2g(g(0,0),\ f(0,0))\cdot {\delta f}(0,0),\\
{\delta b}
&=
{\delta f}(g(0,0),\ f(0,0))
+ \partial_1f(g(0,0),\ f(0,0))\cdot {\delta g}(0,0)\nonumber\\
&\quad{}+ \partial_2f(g(0,0),\ f(0,0))\cdot {\delta f}(0,0).
\end{align*}

\subsection{Equivalent Operator and Fr\'{e}chet Derivative}

The pairs $(g, f)$ forming the two-cycle of C-type correspond to maps with critical points at the origin associated with a singularity of the Jacobian determinant of fold-type, canonically $(x,y)\mapsto(x^2,y)$, and can thus be written as a pair functions in $x^2$ and $y$~\cite{kuznetsov_perioddoubling_1997,kuznetsov_multiparameter_2005,kuznetsov_birth_2008}.
We therefore make the ansatz
\begin{align*}
g(x,y)&=G(x^2,y)=:G(X,Y),\qquad
f(x,y)=F(x^2,y)=:F(X,Y),
\end{align*}
where $(X,Y):=(x^2,y)$, which induces an equivalent operator $T$ given by
\begin{equation*}
T(G, F)(X, Y) = \left(\begin{array}{l}
a^{-1}G(G(a^2X,bY)^2,F(a^2X,bY))\\
b^{-1}F(G(a^2X,bY)^2,F(a^2X,bY))\\
\end{array}\right),
\end{equation*}
for which
\begin{equation*}
R(g,f)(x,y)=T(G,F)(X,Y).
\end{equation*}
In the above, we therefore have $a=G(G(0,0)^2,F(0,0))$, $b=F(G(0,0)^2,F(0,0))$.

The 2-cycle of $R$ conjectured by Kuznetsov et al. corresponds to a 2-cycle of the above operator $T$.
If $(G_1,F_1),(G_2,F_2)$ forms a period-2 cycle of $T$, then the constituent maps must satisfy the following simultaneous functional equations
\begin{align}
    a_1^{-1}G_1\big(G_1(a_1^2X, b_1Y)^2,F_1(a_1^2X, b_1Y)\big)
    &= G_2(X,Y)\quad\mbox{on $\Omega_2$},\label{eqn:fnalt1}\\
    b_1^{-1}F_1\big(G_1(a_1^2X, b_1Y)^2,F_1(a_1^2X, b_1Y)\big)
    &= F_2(X,Y)\quad\mbox{on $\Omega_2$},\label{eqn:fnalt2}\\
    a_2^{-1}G_2\big(G_2(a_2^2X, b_2Y)^2,F_2(a_2^2X, b_2Y)\big)
    &= G_1(X,Y)\quad\mbox{on $\Omega_1$},\label{eqn:fnalt3}\\
    b_2^{-1}F_2\big(G_2(a_2^2X, b_2Y)^2,F_2(a_2^2X, b_2Y)\big)
    &= F_1(X,Y)\quad\mbox{on $\Omega_1$}.\label{eqn:fnalt4}
\end{align}
As above, we will denote the solution maps to the above functional equations by $(G_1^*,F_1^*,G_2^*,F_2^*)$ and recall that the values for the corresponding scaling constants will be denoted by $a_1^*,b_1^*,a_2^*,b_2^*$.

We use the following notation for the rescaled maps
\begin{align*}
    \widetilde{G} &= G(a^2X,bY),\qquad
    \widetilde{F} = F(a^2X,bY).
\end{align*}
The Fr\'{e}chet derivative of $T$ is then given formally as follows:
\begin{align*}
&\mathfrak{D}T(G,\ F)(\delta G,\ \delta F)(X,\ Y)
={}\nonumber\\
&\begin{pmatrix}
-a^{-2}\delta a\cdot G(\widetilde{G}^2, \widetilde{F}) + a^{-1}\Big( \delta G(\widetilde{G}^2, \widetilde{F}) + \partial_1 G(\widetilde{G}^2, \widetilde{F})\cdot 2\widetilde{G}\cdot\delta\widetilde{G} + \partial_2 G(\widetilde{G}^2, \widetilde{F})\cdot \delta \widetilde{F} \Big)\\
-b^{-2}\delta b\cdot F(\widetilde{G}^2, \widetilde{F}) + b^{-1}\Big( \delta F(\widetilde{G}^2, \widetilde{F}) + \partial_1 F(\widetilde{G}^2, \widetilde{F})\cdot 2\widetilde{G}\cdot\delta\widetilde{G} + \partial_2 F(\widetilde{G}^2, \widetilde{F})\cdot \delta \widetilde{F} \Big)
\end{pmatrix}\label{eqn:frechett1}
\end{align*}
where
\begin{align*}
\delta\widetilde{G}
&= \delta G(a^2X,bY) + \partial_1 G(a^2X,bY)\cdot 2a \cdot \delta a \cdot X + \partial_2 G(a^2X,bY)\cdot \delta b \cdot Y,\\
\delta \widetilde{F}
&= \delta F(a^2X,bY) + \partial_1 F(a^2X,bY)\cdot 2a \cdot \delta a \cdot X + \partial_2 F(a^2X,bY)\cdot \delta b \cdot Y,
\end{align*}
and
\begin{align}
\delta a
&= \delta G(G(0,0)^2,F(0,0)) + \partial_1 G(G(0,0)^2,F(0,0))\cdot 2G(0,0)\cdot\delta G(0,0)\nonumber\\
&\quad{}+\partial_2 G(G(0,0)^2,F(0,0))\cdot \delta F(0,0),\\
\delta b
&= \delta F(G(0,0)^2,F(0,0))+ \partial_1 F(G(0,0)^2,F(0,0))\cdot 2G(0,0)\cdot\delta G(0,0)\nonumber\\
&\quad{}+\partial_2 F(G(0,0)^2,F(0,0))\cdot \delta F(0,0).
\label{eqn:frechettn}
\end{align}

\subsection{The conjectured two cycle}

Kuznetsov et al.~\cite{kuznetsov_perioddoubling_1997,kuznetsov_multiparameter_2005,kuznetsov_birth_2008} conjectured, and provided strong numerical evidence for, fixed points of the quadrupling transformation $R^2:=R\circ R$ corresponding to a particular two-cycle for $R$ (hence $T$) with the following values for the state-space scaling constants.
Letting
$\alpha_1^*=1/a_1^*$,
$\alpha_2^*=1/a_2^*$,
$\beta_1^*=1/b_1^*$,
and $\beta_2^*=1/b_2^*$,
the relevant factors are quoted as~\cite{kuznetsov_perioddoubling_1997}
$\alpha^*:=\alpha_1^*\alpha_2^*\approx 6.565350$,
$\beta^*:=\beta_1^*\beta_2^*\approx 22.120227$.
Coefficients for polynomial approximations of the constituent maps are provided in~\cite{kuznetsov_perioddoubling_1997}.
Numerical values for the leading dynamically-relevant eigenvalues of the derivative $\mathfrak{D}(R^2)(G_1,F_1)$ are given by
$\delta_1 \approx 92.43126348$,
$\delta_2 \approx 4.19244418$,
and $\lambda \approx 0.93$. Note that an additional expanding eigenvalue was identified, but determined to correspond to a coordinate-change.

\section{Function spaces}\label{sec:fnspaces}

Here, we will define the function spaces in which we will work. In short, we take the Banach algebra $\mathscr{A}(\Omega_0)$ of functions analytic on the unit polydisc $\Omega_0$, and continuous on its closure $\overline{\Omega_0}$, with finite $\ell_1$-norm of power series coefficients, and make a change of variables $F\mapsto F\circ\psi$ via a map $\psi:{\Omega}\to{\Omega_0}$ to induce an isomorphic Banach algebra of functions analytic on a general polydisc $\Omega$. The induced Banach algebra has a natural sup-norm $\|\cdot\|_{\Omega,\infty}$ isometric to the sup-norm on $\mathscr{A}(\Omega_0)$. However, it is more convenient to use the pullback isometry through $\psi$ of the $\ell_1$-norm on $\mathscr{A}(\Omega_0)$ to provide a norm $\|\cdot\|_{\Omega,1}$ of (weighted) $\ell_1$-type on $\mathscr{A}(\Omega)$. This is useful for bounding the action of linear operators (specifically the Fr\'echet derivatives of the operator $T$). Working in these restricted classes of functions $\mathscr{A}(\Omega)$ also ensures that (for carefully chosen domains) the power series coefficients of the desired maps $G,F\in\mathscr{A}(\Omega)$ decay quickly.

\begin{defn}[The unit polydisc algebra, $\mathscr{A}(\Omega_0)$]
Let $\Omega_0:=D(0,1)\times D(0,1)\subset\mathbb{C}^2$ denote the unit bidisc.
Take
$\widehat{G},\widehat{F}\in\mathscr{A}(\Omega_0)$, the space
of functions analytic on $\Omega_0$
and continuous on its closure $\overline{\Omega_0}$, with finite $\ell_1$-norm of power series coefficients.

Functions $\widehat{H}\in \mathscr{A}(\Omega_0)$ have power series expansions convergent on $\Omega_0$: for $(U,V)\in\Omega_0$ we have
\begin{equation*}
\widehat{H}(U,V)
= \sum_{0\le j,k}^{\infty} \widehat{H}_{jk}\widehat{e}_{jk}(U,V)
:=\sum_{0\le j,k}^{\infty} \widehat{H}_{jk}U^jV^k,
\end{equation*}
with the algebra operations (together with composition) given by
\begin{align*}
{(a\widehat{F}+b\widehat{G})}(U,V)&:=a\widehat{F}(U,V)+b\widehat{G}(U,V),\\
{(\widehat{F}\cdot\widehat{G})}(U,V)&:=\widehat{F}(U,V)\, \widehat{G}(U,V),\\  {\widehat{H}\circ(\widehat{F}\oplus\widehat{G})}(U,V)&:=\widehat{H}\bigl(\widehat{F}(U,V),\widehat{G}(U,V)\bigr).
\end{align*}
This forms a Banach algebra with Schauder basis $(\widehat{e}_{jk})_{0\le j,k}$ and $\ell_1$-norm of power series coefficients given by
\begin{equation*}
{\|\widehat{H}\|} = \|\widehat{H}\|_1 := \sum_{0\le j,k}^{\infty}|\widehat{H}_{jk}|.
\end{equation*}
\end{defn}
(We note that the composition operation is well-defined in $\mathscr{A}(\Omega_0)$ for $\|\widehat{G}\|_1,\|\widehat{F}\|_1<1$.)

We will now use this structure to induce an isometrically isomorphic Banach algebra of functions analytic on a general polydisc.
\begin{defn}[The general polydisc algebra, $\mathscr{A}(\Omega)$]
Let
\begin{equation*}
\Omega := D(c, r)\times D(d, s)\subset\mathbb{C}^2,
\end{equation*}
and consider the affine map $\psi:\Omega\to\Omega_0$ to the unit bidisc given by
\begin{equation*}
\psi:(X, Y)
\mapsto(U,V)
:=\left(\frac{X-c}{r}, \frac{Y-d}{s}\right).
\end{equation*}
We take $\psi$ as a change of coordinates to define certain maps analytic on $\Omega$ via $\widehat{G}\mapsto\widehat{G}\circ\psi=:G$ and denote the set of such maps $\mathscr{A}(\Omega)$. Thus functions $G\in\mathscr{A}(\Omega)$ are analytic on $\Omega$, and continuous on its closure and have the decomposition $G=\widehat{G}\circ\psi$ for some map $\widehat{G}\in\mathscr{A}(\Omega_0)$.
\end{defn}
Thus
\begin{equation*}
G(X, Y)
= \sum_{0\le j,k} \widehat{G}_{jk}\left(\frac{X-c}{r}\right)^j\left(\frac{Y-d}{s}\right)^k.
\end{equation*}
We note the following induced operations (that are isomorphic to those on $\mathscr{A}(\Omega_0)$, except for the notable additional subtlety involved in composition):
\begin{align}
aG+bF
&=(a\widehat{G}+b\widehat{F})\circ\psi,
\label{eqn:banachiso1}\\
G\cdot F
&= (\widehat{G}\cdot\widehat{F})\circ\psi,
\label{eqn:banachiso2}\\
H\circ (G,F)
&= \left(\widehat{H}\circ\psi\circ(\widehat{G},\widehat{F})\right)\circ\psi.
\label{eqn:banachiso3}
\end{align}
Note further that we have (trivially) an isometry for the sup-norm, in the sense that
\begin{equation*}
\|G\|_{\Omega,\infty}
=\|\widehat{G}\circ\psi\|_{\Omega,\infty}
=\|\widehat{G}\|_{\psi(\Omega),\infty}
=\|\widehat{G}\|_{\Omega_0,\infty}.
\end{equation*}
This leads us to define a natural $\ell_1$-type norm compatible with the sup-norm via the pullback isometry through $\psi$:
\begin{equation*}
\|G\|_{\Omega,1}
=\|\widehat{G}\circ\psi\|_{\Omega,1}
:=\|\widehat{G}\|_{\Omega_0,1}.
\end{equation*}
We denote the resulting Banach algebra by $\mathscr{A}(\Omega)$ and work with the natural Schauder basis $(e_{jk})_{0\le j,k}$ where $e_{jk}:=\widehat{e}_{jk}\circ\psi$.
Note that the resulting norm $\|G\|_{\Omega,1}=\|\widehat{G}\|_1$ is a sub-multiplicatively weighted $\ell_1$-norm with respect to the power series expansion of $G$ itself about $(c, d)$.
(For general results concerning functions of several complex variables analytic on polydisc domains, see~\cite{rudin_function_1969}.)

\section{Method for cycles}\label{sec:cycles}

In what follows, we will work with distinct domains $\Omega_k$ for $k=1,2$ for each pair $(G_k,F_k)$ of functions forming the two-cycle of the operator $T$.
More specifically, we will consider the following domains (bidiscs):
\begin{align*}
\Omega_1 &:= D(c_1, r_1)\times D(d_1, s_1),\qquad
\Omega_2 := D(c_2, r_2)\times D(d_2, s_2),
\end{align*}
and pairs of maps analytic on each domain
\begin{align*}
P_1:=(G_1,F_1)&\in\mathscr{A}(\Omega_1)^2,\qquad
P_2:=(G_2,F_2)\in\mathscr{A}(\Omega_2)^2.
\end{align*}
We then denote the operators between them corresponding to $T$ acting on $\mathscr{A}(\Omega_1)^2$ and on $\mathscr{A}(\Omega_2)^2$ (respectively) by
\begin{align*}
T_1:\mathscr{A}(\Omega_1)^2&\to \mathscr{A}(\Omega_2)^2,\qquad
T_2:\mathscr{A}(\Omega_2)^2\to \mathscr{A}(\Omega_1)^2.
\end{align*}
Thus, acting on the product space, we have
\begin{equation*}
\mathcal{T}:=
\mathrm{diag}(T_1,T_2):\mathscr{A}(\Omega_1)^2\oplus\mathscr{A}(\Omega_2)^2
\to\mathscr{A}(\Omega_2)^2\oplus\mathscr{A}(\Omega_1)^2.
\end{equation*}
We want to find the solutions to
\begin{align*}
P_1-T_2(P_2) &= 0,\qquad
P_2-T_1(P_1) = 0.
\end{align*}
%
%
We therefore seek a root of the operator $\mathcal{F}$ given by
\begin{equation*}
\mathcal{F} = I-S\mathcal{T},
\end{equation*}
where we define $S:\mathscr{A}(\Omega_1)^2\oplus \mathscr{A}(\Omega_2)^2\leftrightarrow \mathscr{A}(\Omega_2)^2\oplus \mathscr{A}(\Omega_1)^2$ so that $S$  swaps pairs: 
\begin{equation*}
S
=\left(\begin{array}{ll}
0 &I\\
I &0
\end{array}\right).
\end{equation*}
Note that $\mathcal{F}$ is a self-map of the direct sum $\mathscr{A}(\Omega_1)^2\oplus\mathscr{A}(\Omega_2)^2$. Later (in Section~\ref{sec:results}) we will make use of the two natural equivalent norms on the direct sum given by
\begin{equation*}
\|(G_1,F_1,G_2,F_2)\|_{\mathrm{sum},1}
:=\|G_1\|_{\Omega_1,1}
+\|F_1\|_{\Omega_1,1}
+\|G_2\|_{\Omega_2,1}
+\|F_2\|_{\Omega_2,1},
\end{equation*}
and
\begin{equation*}
\|(G_1,F_1,G_2,F_2)\|_{\mathrm{max},1}
:=\max\left\{\|G_1\|_{\Omega_1,1},
\|F_1\|_{\Omega_1,1},
\|G_2\|_{\Omega_2,1},
\|F_2\|_{\Omega_2,1}\right\}.
\end{equation*}
The Fr\'echet derivative of $\mathcal{F}$ is given by 
\begin{equation}
\mathfrak{D}\mathcal{F}(V)
= I-S \mathfrak{D}\mathcal{T}(V),\label{eqn:frechetcurlyt}
\end{equation}
since $S$ is a linear operator. In detail,
\begin{align*}
\mathfrak{D}\mathcal{F}(P_1,P_2)(\delta P_1, \delta P_2)
&= (\delta P_1-\mathfrak{D}T_2(P_2)\delta P_2,\ \delta P_2-\mathfrak{D}T_1(P_1) \delta P_1).
\end{align*}

We now take an initial guess for the maps of the two-cycle and improve it using the Newton method for $\mathcal{F}$.


\begin{figure}
\centering
\begin{tabular}{cc}
\includegraphics[width=0.45\textwidth,trim={0.1in 0in 0in 0in},clip]{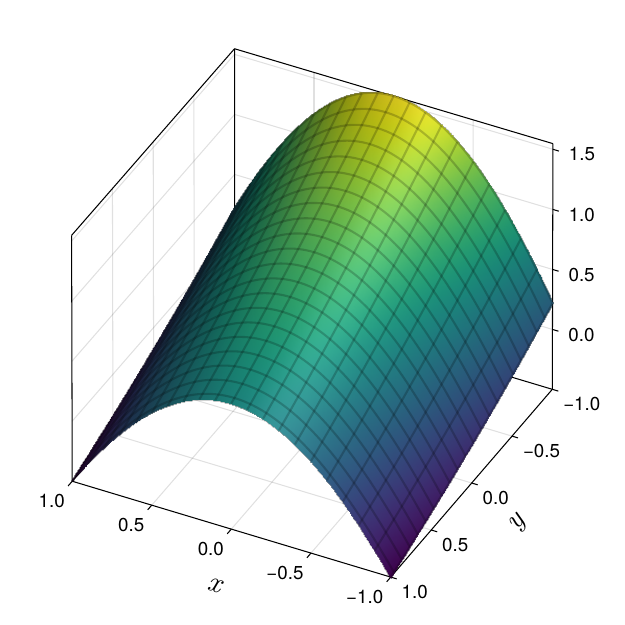}&
\includegraphics[width=0.45\textwidth,trim={0.1in 0in 0in 0in},clip]{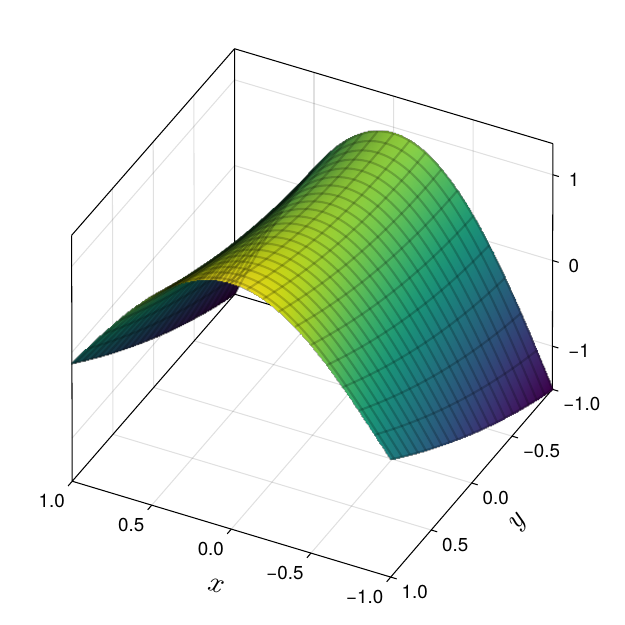}\\
(a) $g_1(x, y)$&
(b) $f_1(x, y)$\\
\includegraphics[width=0.45\textwidth,trim={0.1in 0in 0in 0in},clip]{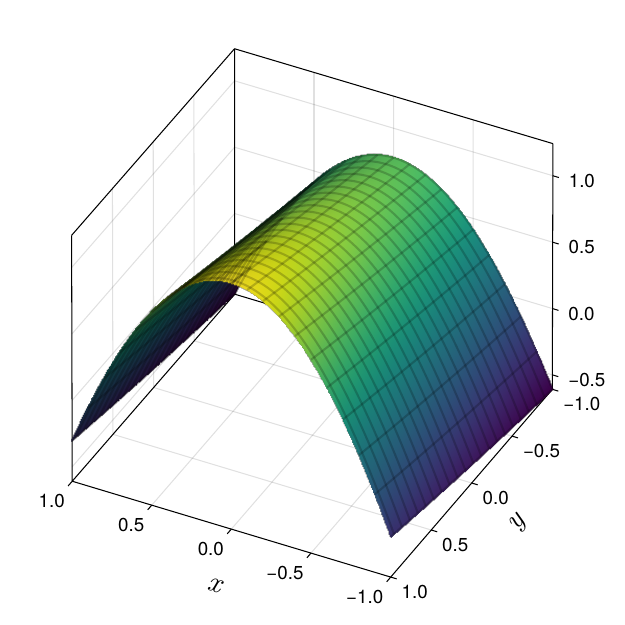}&
\includegraphics[width=0.45\textwidth,trim={0.1in 0in 0in 0in},clip]{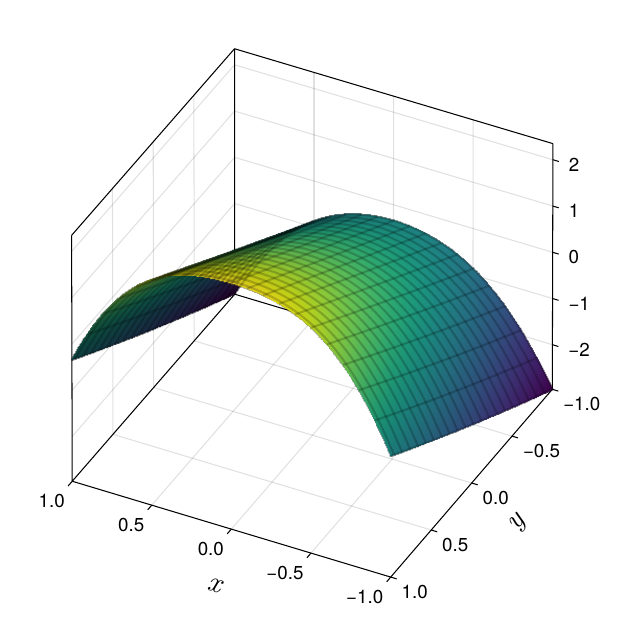}\\
(c) $g_2(x, y)$&
(d) $f_2(x, y)$
\end{tabular}
\caption{The constituent maps of the RG two-cycle for $R$: (a) $g_1(x,y)$, (b) $f_1(x,y)$, (c) $g_2(x, y)$, (d) $f_2(x, y)$, computed using $g_k^0(x,y):=G_k^0(x^2,y)$ and $f_k^0(x,y):=F_k^0(x^2,y)$. (Note the reversal in the $y$-axis direction, chosen for clarity.)}
\label{fig:goodapproxsmall}
\end{figure}

\begin{figure}
\centering
\begin{tabular}{cc}
\includegraphics[width=0.45\textwidth,trim={0.2in 0in 0in 0.1in},clip]{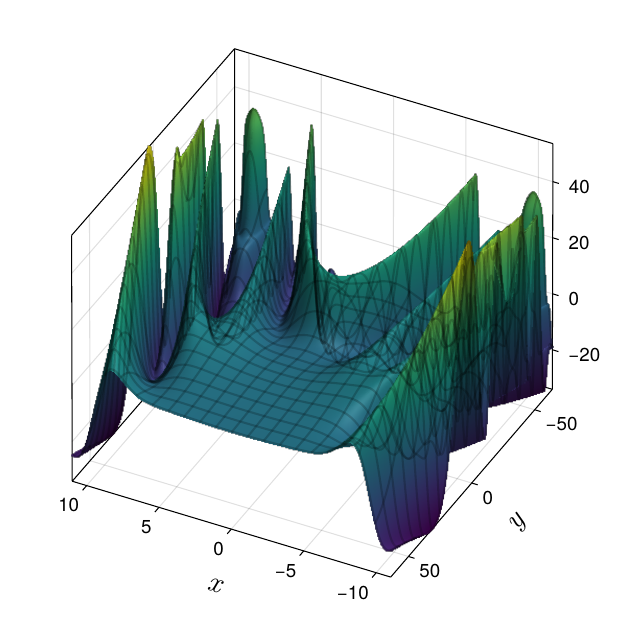}&
\includegraphics[width=0.45\textwidth,trim={0.2in 0in 0in 0.1in},clip]{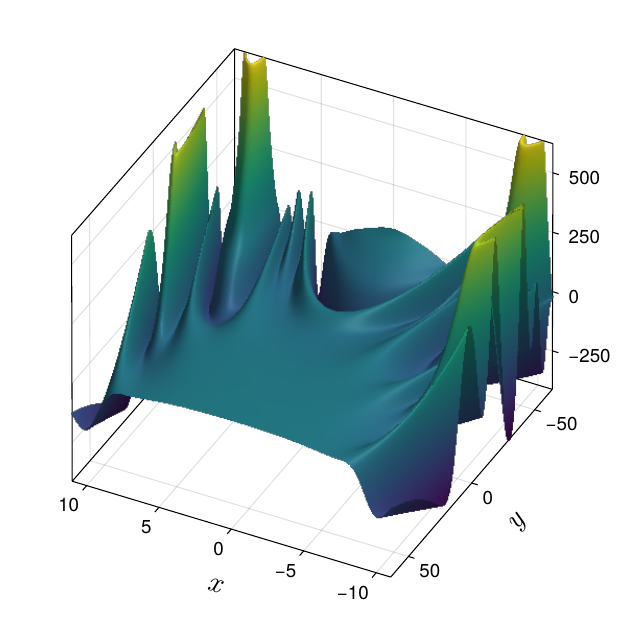}\\
(a) $g_1(x, y)$
& (b) $f_1(x, y)$\\
\includegraphics[height=0.45\textwidth,trim={0.2in 0in 0in 0.1in},clip]{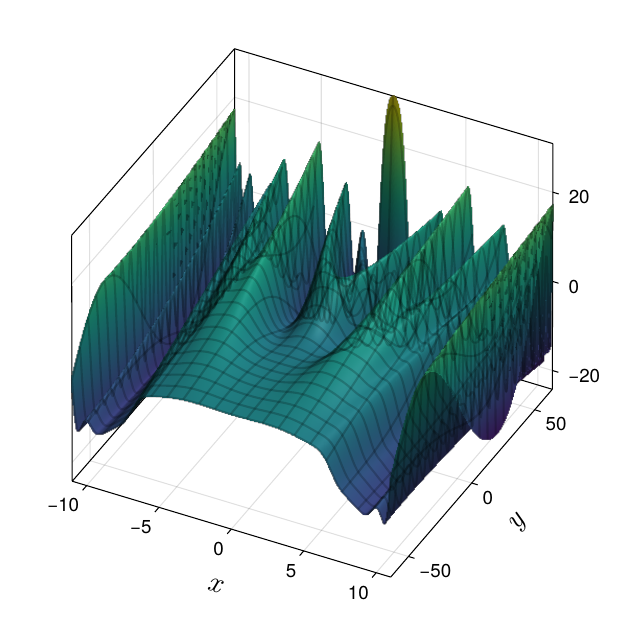}&
\includegraphics[height=0.45\textwidth,trim={0.2in 0in 0in 0.1in},clip]{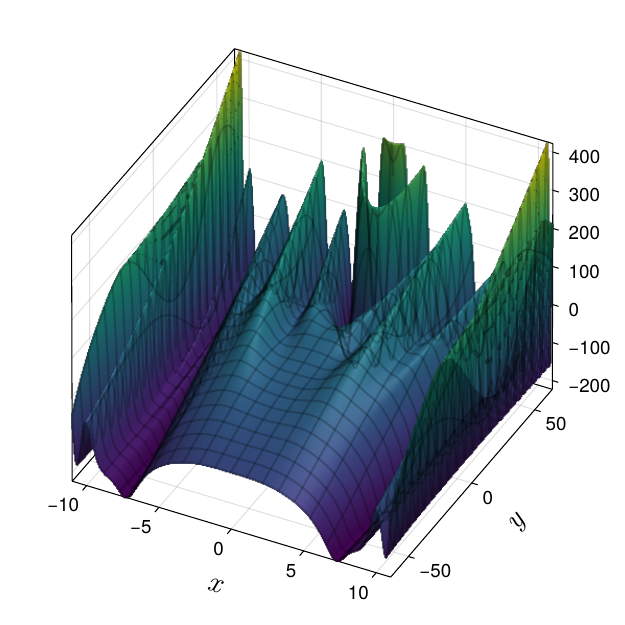}\\
(c) $g_2(x, y)$
& (d) $f_2(x, y)$
\end{tabular}
\caption{The constituent maps $(g_1,f_1),(g_2,f_2)$ of the RG two-cycle of $R$, computed using $g_k(x,y)=G_k^0(x^2,y)$ and $f_k(x,y)=F_k^0(x^2,y)$ for $(x^2,y)\in\Omega_k$ as a base case, together with recurrence relations derived from the functional equations~(\ref{eqn:fnalt1})--(\ref{eqn:fnalt4}) for a two-cycle of $T$. (Note the reversal in the $y$-axis direction, chosen for clarity, in (a),(b).)}
\label{fig:goodapproxgfrecur}
\end{figure}

\subsection{Approximate Newton Method}\label{sec:approxcycle1}

In this section, we show how we obtain good approximations to the maps $(G_k,F_k)$ of the C-type two-cycle by fixing a maximum homogeneous polynomial degree $N$ and computing  polynomial approximations via a (truncated) Newton method, after which suitable domains $\Omega_k$ are chosen such that the operators $T_k$ are well-defined and differentiable (with compact derivatives) on the corresponding spaces $\mathscr{A}(\Omega_k)$. The approximations of $(G_k,F_k)$ are then expanded with respect to these new domains and the Newton method is repeated, working in the polynomial part of the space $\mathscr{A}(\Omega_1)^2\oplus\mathscr{A}(\Omega_2)^2$.

We wish to implement the Newton operator
\begin{equation*}
V \mapsto V - [\mathfrak{D}\mathcal{F}(V)]^{-1}\mathcal{F}(V).
\end{equation*}
Let
\begin{equation*}
\varphi_{1,\Omega}:\mathcal{P}\mathscr{A}(\Omega)
\to\mathbb{R}^{N_1},
\end{equation*}
represent the coefficient isomorphism between the polynomial part of the space $\mathscr{A}(\Omega)$ and the corresponding truncation of the sequence space $\ell_1$ for one map. 
Similarly, let
\begin{equation*}
\varphi_{4}:\mathcal{P}\left(\mathscr{A}(\Omega_1)^2\oplus \mathscr{A}(\Omega_2)^2\right)
\to\mathbb{R}^{N_4},
\end{equation*}
represent the coefficient isomorphisms for a 
pair of pairs of maps. 
We can then approximate $\mathfrak{D}\mathcal{F}(V)$ via the matrix elements
\begin{equation*}
M_{jk,mn}:=\mathcal{P}[\mathfrak{D}\mathcal{F}(e_{mn})]_{jk},
\end{equation*}
in which $\mathcal{P}$ denotes the projection to the polynomial part of the space,
and the approximate Newton method can therefore be implemented via
\begin{equation*}
V
\mapsto V - \left(\varphi_{4}^{-1}M^{-1}\varphi_{4}\right)(\mathcal{F}(V)).
\end{equation*}
Note that, for the proof described in this paper, we took $0\le j+k\le N=24$, thus $N_1=(N+1)(N+2)/2=325$, $N_2=2N_1=650$, $N_4=2N_2=1300$, and the  derivatives of the operators $T_1,T_2$ on the polynomial parts of the spaces are represented by $N_2^2=422,500$ matrix elements, with $N_4^2=1,690,000$ matrix elements for the polynomial part of $\mathfrak{D}\mathcal{F}(V)$. 
A starting point for polynomial approximation was taken from~\cite{kuznetsov_perioddoubling_1997} (Table~2, p.255).

\subsection{Finding a good approximate two-cycle}

In order to move from polynomials to analytic maps, in Section~\ref{sec:de} and Section~\ref{sec:domaindetails} (motivated by Section~\ref{sec:rigorousframework}) we address the issue of choosing good domains $\Omega_k$ on which to define the relevant function spaces.
Once good domains have been chosen for our functions, we expand the approximate two-cycle polynomials with respect to these new domains and then re-apply the numerical Newton method (working in $\mathscr{A}(\Omega)$) to improve them, resulting in a \emph{good approximate two-cycle},
\begin{equation*}
V^0=(P_1^0,P_2^0)=((G_1^0,F_1^0),(G_2^0,F_2^0)).
\end{equation*}
In order to expand the approximations with respect to the new domains, recall the representation for analytic maps chosen in Section~\ref{sec:fnspaces} and note that if $G=\widehat{G}\circ\psi$ then $\widehat{G}=G\circ\psi^{-1}$ where $\psi^{-1}:(U,V)
\mapsto(X,Y)
:=(c+rU, d+sV)$.
which means that the representations (and norms) of the identity map, as an element of the spaces $\mathscr{A}(\Omega_k)$, and hence the individual coordinate projections,
depend on the choice of domain $\Omega$.

The constituent maps $(g_1,f_1),(g_2,f_2)$ of the resulting approximate two-cycle for $R$ are visualised in Figure~\ref{fig:goodapproxsmall} for comparison with~\cite{kuznetsov_birth_2008} and (over a much larger domain) in Figure~\ref{fig:goodapproxgfrecur} using $g_k(x,y)=G_k(x^2,y)$ and $f_k(x,y)=F_k(x^2,y)$.
To evaluate the corresponding maps $(G_1,F_1),(G_2,F_2)$ of the resulting approximate two-cycle for $T$, over much larger domains than $\Omega_1,\Omega_2$, we use the functional equations~(\ref{eqn:fnalt1})--(\ref{eqn:fnalt4}) to provide recurrence relations, using the maps $G_k^0,F_k^0$ as the base case when $(X,Y)\in\Omega_k$. (The corresponding recursions terminate due to the \emph{real domain extension} property that will be verified in Section~\ref{sec:de} and illustrated in Figure~\ref{fig:de_real}.) 

\section{Domain extension}\label{sec:de}

The domains $\Omega_k$ must be chosen carefully to ensure that the operators $T_1$ and $T_2$ are well defined on the spaces $\mathscr{A}(\Omega_1)^2$ and $\mathscr{A}(\Omega_2)^2$, respectively.
We have the 
functional equations~(\ref{eqn:fnalt1})--(\ref{eqn:fnalt4}).
We therefore wish to find domains $\Omega_1,\Omega_2$ for which the following inclusions (`domain extension conditions') hold:
\begin{align}
\overline{\mbox{diag}(a_1^2, b_1)\,\Omega_2}
&\subset\Omega_1,\label{eqn:tde1}
\qquad
\overline{(G_1^2,F_1)\,\mbox{diag}(a_1^2, b_1)\,\Omega_2}
\subset\Omega_1,\\ 
\overline{\mbox{diag}(a_2^2, b_2)\,\Omega_1}
&\subset\Omega_2,
\qquad
\overline{(G_2^2,F_2)\,\mbox{diag}(a_2^2, b_2)\,\Omega_1}
\subset\Omega_2.\label{eqn:tde4}
\end{align} 
(In the above, the overline denotes topological closure.) Later, in Section~\ref{sec:domaindetails}, we will impose additional constraints on the choice of domains, in order that the results of compositions of maps remain within the restricted classes $\mathscr{A}(\Omega_k)$ and that suitable error bounds can be maintained during compositions of functions in the computer-assisted part of the proof (Sections~\ref{sec:domaincentres}--\ref{sec:domainradii}). Moreover, by insisting that the closures of the left-hand sides in~(\ref{eqn:tde1})--(\ref{eqn:tde4}) are contained strictly within the relevant domains, we ensure that the maps have analytic extensions to strictly larger domains and that the derivatives $\mathfrak{D}T_1(P_1)$ and $\mathfrak{D}T_2(P_2)$ are compact.


\subsection{Numerical complex domain extension}

If the domain extension conditions hold for complex domains, for rescaling and composition operators such as $T$ and $R$, then it follows that the Fr\'echet derivative is compact for the operator acting on the corresponding Banach space of analytic functions. (This is discussed in more detail in~\cite{mestel_computer_1985,mackay_renormalisation_1993,burbanks_rigorous_2021a}.)

Here, we provide a parameterisation of the (closed) polydisc $\overline{\Omega}$ where $\Omega=D(c,r)\times D(d,s)$ that will be useful in illustrating domain extension:
\begin{equation*}
[0,1]^2\times[0,2\pi)^2
\ni(\rho_1, \rho_2; \theta_1, \theta_2)
\mapsto (c+\rho_1 r e^{i\theta_1}, d+\rho_2 s e^{i\theta_2})
\in\mathbb{C}^2.
\end{equation*}
We may write the boundary $\partial\Omega$ as a one-parameter family of $2$-tori given by
\begin{align*}
\Pi&:
[-1,+1]\times[0,2\pi)^2\to\partial\Omega\subset\mathbb{C}^2,\\
\Pi&:
(p, \theta_1, \theta_2)
\mapsto (c+\min(1, 1-p)re^{i\theta_1}, d+\min(1, 1+p)se^{i\theta_2}),
\end{align*}
consisting of the union of two filled $2$-tori,
\begin{equation*}
\partial\Omega
=(\overline{D(c,r)}\times\partial{D(d,s)})
\cup
(\partial{D(c,r)}\times\overline{D(d,s)}),
\end{equation*}
(corresponding to $p\in[-1,0]$ and $p\in[0,1]$ respectively) foliated by a one-parameter family of $2$-tori, that intersect along their common boundary, the \emph{distinguished boundary} $\flat\Omega\subset\partial\Omega$  ($p=0$) at which $\partial\Omega$ is non-smooth:
\begin{equation*}
\flat\Omega
=\Pi\bigl(0, [0,2\pi), [0,2\pi)\bigr)
=\partial D(c,r)\times\partial D(d,s)\cong\mathbb{T}^2.
\end{equation*}
The parameterisation has singular points at $p=\pm 1$ corresponding to the degenerate tori (circles) $\partial D(c,r)\times\{d\}$ and $\{c\}\times\partial D(d,s)$.
We visualise images of the bidisc under a map by sampling the parameter $p$ and plotting a discretisation of a number of the resulting $2$-tori by sampling values $(\theta_1,\theta_2)$ for the angular variables. For clarity, we include the degenerate tori and the distinguished boundary.
In Figure~\ref{fig:de} of Section~\ref{sec:domaindetails}, we show domains for which complex domain extension holds (with the restriction to the corresponding real domains $\Xi_1,\Xi_2$ shown in Figure~\ref{fig:de_real}(b) for comparison).
We note that the maximum principle for several complex variables implies that it is enough to verify that the images of the distinguished boundary lie within the corresponding polydiscs~\cite{rudin_function_1969}.
The computer-assisted portion of the proof verifies domain extension by computing a rigorous enclosure of the image of the polydisc under the various maps.

\section{Rigorous computations}\label{sec:rigorousframework}

In order to bound the results of computations on the spaces $\mathscr{A}(\Omega_k)$, we adapt the frameworks described by Eckmann, Koch, and Wittwer~\cite{eckmann_computer-assisted_1984,eckmann_computer_1985} for analytic maps of one and two variables. In short, we fix a (homogeneous) truncation degree $N$ and partition spaces into the corresponding polynomial part and high-order part. Bounds are maintained carefully on the power series coefficients of the polynomial parts and on the norms of the high-order parts.
(Such methods have their roots in Moore's interval arithmetic~\cite{moore_interval_1966}, described more recently in greater abstraction by Kaucher et al~\cite{kaucher_self-validating_1984,kulisch_arithmetic_1986}.) A central concept is the \emph{function ball} defined by bounds, comprising a vector $v_P=([a_{jk},b_{jk}])_{0\le j+k\le N}$ of (proper) closed intervals for polynomial coefficients and two upper bounds, $v_H,v_E\ge 0$, on norms: one for the high-order part, and another for the general or `error' part. Incorporating the error part enables us to describe a ball of functions in the usual sense, and provides greater flexibility in absorbing errors accumulated in computations.
\begin{defn}[Function ball]\label{defn:fnball}
Define the (convex closed) function ball \(\mathscr{B}({v_P},{v_H}, {v_E})\subset \mathscr{A}(\Omega)\) by
\begin{align*}
\mathscr{B}({v_P}, {v_H}, {v_E}) :=
\bigl\{
&\ F\in \mathscr{A}(\Omega):\\
&\ F = {F_P}+{F_H}+{F_E},\\
&\ {F_P\in \mathcal{P}\mathscr{A}(\Omega)},\ {F_H\in \mathcal{H}\mathscr{A}(\Omega)},
\ {F_E\in \mathscr{A}(\Omega)},\\
&\ {F_P(X,Y) = \sum_{\substack{0\le j,k\\ 0\le j+k\le N}}F_{jk}e_{jk}(X,Y),\ F_{jk}\in[a_{jk},b_{jk}]},\\
&\ {\|F_H\|\le v_H},
\ {\|F_E\|\le v_E}
\bigr\}.
\end{align*}\end{defn}
In the above, $\mathcal{P}$ denotes the canonical projection to the polynomial part of the space:
for $(X,Y)\in\Omega$ and
\begin{equation*}
F(X,Y) = \sum_{0\le j,k}F_{jk}e_{jk}(X,Y),
\end{equation*}
we have $\mathcal{P}F(X,Y)$ denoting those terms in the series for which $j+k\le N$,
and $\mathcal{H}:=I-\mathcal{P}$ denotes the canonical projection to the high-order part of the space. In the implementation, we use computer-representable numbers for the bounds $(v_P,{v_H},{v_E})$ and interval arithmetic with the relevant directed-rounding modes.

Given $H,G,F\in\mathscr{A}(\Omega)$, and scalar $a\in\mathbb{C}$, detailed bounds for the resulting \emph{function ball algebra} (specifying function balls guaranteed to give rigorous enclosures of the results) are formulated in~\cite{eckmann_computer-assisted_1984,eckmann_computer_1985} for maps of one and two variables, for the operations that result in $F+G$, $aF$, $F\cdot G$, $H\circ (G,F)$, $\|F\|$, and evaluation $F(X,Y)$. We note that although the derivative operator that results in $\partial_jF$ is unbounded, it is possible to bound expressions of the form $\partial_jH\circ (G,F)$, including those appearing in the Fr\'echet derivatives of $T_k$~(\ref{eqn:frechett1})--(\ref{eqn:frechettn}) in Section~\ref{sec:rgops}, given suitable assumptions on $G,F$. 

Bounds for function balls in $\mathscr{A}(\Omega)$, for a general polydisc $\Omega$, are readily computed in terms of those for `standard' function balls in $\mathscr{A}(\Omega_0)$ by using the relations~(\ref{eqn:banachiso1})--(\ref{eqn:banachiso3}).
With a slight abuse of notation, where the polynomial part $F_P$ is known (and thus the intervals $[a_{jk},b_{jk}]$ of $v_P$ are degenerate), we use the notation $\mathscr{B}(F_P;v_H,v_E)$, to denote the ball of functions centered on that particular polynomial.

In order to motivate our discussion of domains in Section~\ref{sec:domaindetails}, we must outline briefly the method used to bound compositions of maps in this framework.

\subsection{Bounds for composition}\label{sec:composition}

In this section, we examine the computational framework (adapted from~\cite{eckmann_computer_1985,eckmann_computer-assisted_1984}) used to bound compositions of the form
\begin{equation*}
(h, g, f) \mapsto h\circ(g, f),    
\end{equation*}
and partial (coordinate) derivatives followed by composition $(h, g, f) \mapsto (\partial_jh)\circ(g, f)$,
such expressions appearing in the definition of the operators $T_k$ and in the formal Fr\'echet derivative $\mathfrak{D}T_k(P_k)$ taken at a pair $P_k=(G_k,F_k)$ given in~(\ref{eqn:frechett1})--(\ref{eqn:frechettn}) and hence in the derivatives $\mathfrak{D}\mathcal{T}(V)$ and $\mathfrak{D}\mathcal{F}(V)$~(\ref{eqn:frechetcurlyt}) for $V=(P_1,P_2)$.

We will work in $\mathscr{A}(\Omega_0$) (the generalisation to $g,f\in\mathscr{A}(\Omega_k)$, $g,f:\Omega_k\to\Omega_j$ and $h\in\mathscr{A}(\Omega_j)$ is straightforward).
Let $h\in\mathscr{B}(w_P,w_H,w_E)$, $g\in\mathscr{B}(v_P,v_H,v_E)$, and $f\in\mathscr{B}(u_P,u_H,u_E)$. Thus we can write $h=h_P+h_H+h_E$, $g=g_P+g_H+g_E$, and $f=f_P+f_H+f_E$ for suitably-bounded functions $h_P,g_P,f_P\in\mathcal{P}\mathscr{A}(\Omega_0)$, $h_H,g_H,f_H\in\mathcal{H}\mathscr{A}(\Omega_0)$, and $h_E,g_E,f_E\in\mathscr{A}(\Omega_0)$.
Firstly, we make use of linearity of composition in the leftmost argument:
\begin{equation}
(h_P+h_H+h_E)\circ(g,f)
=h_P\circ(g,f)
+h_H\circ(g,f)
+h_E\circ(g,f),\label{eqn:compositionterms}
\end{equation}
and note that the first term $h_P\circ(g,f)$ in~(\ref{eqn:compositionterms}) is a polynomial in $g,f$ that may be bounded by implementing operations to compute rigorous enclosures of the corresponding elementary Banach algebra operations $(a,G)\mapsto aG$, $(G, F)\mapsto G+F$, and $(G,F)\mapsto G\cdot F$. For $\|g\|_1,\|f\|_1<1$, which may be verified by computing rigorous enclosure of the norm, $G\mapsto\|G\|_1$, bounds on the third term in~(\ref{eqn:compositionterms}) follow from
\begin{equation*}
\|h_E\circ(g, f)\|_1\le w_E.
\end{equation*}
For the second term in~(\ref{eqn:compositionterms}), note that composition of a high-order function with a pair of functions can contribute to both polynomial and high-order terms, but that the result will be strictly high-order if $g,f$ have zero constant terms.
To take this into account, we will decompose $h_H \circ (g,f)$ into 2 parts, one will contribute to the high order part of the result and the other to the error part (in general, the latter may contain polynomial terms, but we lack information about individual coefficients).
Write $g=g_{00}+g_*$ and $f=f_{00}+f_*$ where $g_{00}$, $f_{00}$ are the constant terms of $g$, $f$:
\begin{equation*}
    h_H \circ (g,f)
    = \sum_{j+k\ge N+1} h_{jk} \big( g_{00} + g_* \big)^j \big( f_{00} + f_* \big)^k.
\end{equation*}
Terms that involve $g_*,f_*$ only (and are therefore strictly high-order) are bounded by
\begin{equation*}
\left\Vert \cdot \right\Vert_1 \leq w_H\cdot \max\left( \left\Vert g_* \right\Vert_1, \left\Vert f_* \right\Vert_1 \right)^{N+1},
\end{equation*}
which we absorb into the high-order bound on $h_H \circ (g,f)$.
For the other terms, note
\begin{equation*}
\left\Vert g^j  f^{k} - g_*^j f_*^{k}   \right\Vert_1
\leq \left\Vert g \right\Vert_1 ^j \left\Vert f \right\Vert_1 ^{k} - \left\Vert g_* \right\Vert_1^j \left\Vert f_* \right\Vert_1^{k},
\end{equation*}
for $j+k=N+1$, which follows from submultiplicativity of the $\ell_1$-norm and additivity on disjoint supports; $\Vert g\Vert_1=\vert g_{00}\vert+\Vert g_*\Vert_1$ and $\Vert f\Vert_1=\vert f_{00}\vert+\Vert f_*\Vert_1$.
Summing over the relevant terms in the series  contributes the bound
\begin{equation*}
\left\Vert \cdot \right\Vert_1 \leq w_H\cdot \max_{j+k=N+1}\left( \left\Vert g \right\Vert_1^j \left\Vert f \right\Vert_1^k - \left\Vert g_* \right\Vert_1^j \left\Vert f_* \right\Vert_1^k \right),
\end{equation*}
to the error term of the result.
The size of the bracketed expression is controlled by the constant terms on $g$, $f$.
For the current paper, this demands a careful choice of distinct domains $\Omega_1,\Omega_2$ for the pairs $(G_k,F_k)$ rather than a common domain.
In Section~\ref{sec:domaindetails} we will show that a good choice for the centres of the domains is to take particular
points that are exchanged by the `inner' maps in the functional equations for the two-cycle~(\ref{eqn:fnalt1})--(\ref{eqn:fnalt4}).

\section{Finding good domains}\label{sec:domaindetails}

Recall that we imposed `domain extension' conditions in Section~\ref{sec:de} to ensure that the operators $T_k$ are well-defined on $\mathscr{A}(\Omega_k)^2$; we now impose additional constraints.
The centres and radii determine the norms on the resulting spaces $\mathscr{A}(\Omega_k)$. Good choices lead to faster decay of power series coefficients for our functions, allowing a proof with a lower truncation degree, and ensure that the error terms involved in a composition $H\circ(G,F)$ remain manageable.

\subsection{Domain Centres}\label{sec:domaincentres}

In what follows, we first seek suitable locations for the centres $c_k,d_k$ of the bidiscs $\Omega_k=D(c_k,r_k)\times D(d_k,s_k)$.
Consider the functional equations~(\ref{eqn:fnalt1})--(\ref{eqn:fnalt4}), we claim that suitable centres $(c_k,d_k)$ correspond to pairs of points in $\mathbb{C}^2$ that are exchanged by the inner maps and are therefore roots of the outer maps.
Firstly, note that the dominant contribution to the error bounds during the operations of composition (section~\ref{sec:composition}) and derivative followed by composition, of general functions $F,G,H\in\mathscr{A}(\Omega_0)$, are due to the presence of a nonzero constant term on the functions $G,F$ in expressions of the form
$H\circ(G\oplus F)$ and $\partial_kH\circ(G\oplus F)$.

The functional equations~(\ref{eqn:fnalt1})--(\ref{eqn:fnalt4}) for the RG two-cycle impose a symmetry on the functions $(G_1,F_1),(G_2,F_2)$ that constrains certain properties of `inner' maps involved in the compositions.
%
%
Suppose that the pairs $(G_1,F_1)$, $(G_2,F_2)$ form a two-cycle of $T$, and the pairs $(c_1,d_1)$, $(c_2,d_2)$ are exchanged by the `inner' maps $\Theta_1$ and $\Theta_2$:
\begin{align*}
    (c_1,d_1)&=\Theta_1(c_2,d_2):=(G_1(a_1^2 c_2,\,b_1 d_2)^2,\,F_1(a_1^2 c_2,\,b_1 d_2)),\\
    (c_2,d_2)&=\Theta_2(c_1,d_1):=(G_2(a_2^2 c_1,\,b_2 d_1)^2,\,F_2(a_2^2 c_1,\,b_2 d_1)).
\end{align*}
We choose domain centres $(c_{1}, d_{1})\in\mathbb{C}^2$ and  $(c_{2}, d_{2})\in\mathbb{C}^2$ as above. Let $\psi_k:\Omega_k\to\Omega_0$ via
$\psi_k:(X, Y) \mapsto\left( (X - c_{k})/r_{k}, (Y - d_{k})/s_{k} \right)$
and note that $\psi_k:(c_k,d_k)\mapsto(0,0)$.


For the outer compositions of the functions $G_k,F_k$ in the expressions for the operators $T_k$, we have on $\Omega_1$:
\begin{align}
(G_2,F_2)\circ(G_2^2,F_2)\circ\mathrm{diag}(a_2^2,b_2)
&=
(\widehat{G}_2,\widehat{F}_2)\circ\left[\psi_2
    \circ\Theta_2\circ\psi_1^{-1}\right]\circ\psi_1.\label{eqn:outernorm1a}
\end{align}
It follows immediately that the constant terms on the pair of maps (in $\mathscr{A}(\Omega_0)^2$) represented by the square-bracketed expression are zero.
Similarly, on $\Omega_2$:
\begin{align}
(G_1,F_1)\circ(G_1^2,F_1)\circ\mathrm{diag}(a_1^2,b_1)
&= 
(\widehat{G}_1,\widehat{F}_1)\circ\left[\psi_1
    \circ\Theta_1\circ\psi_2^{-1}\right]\circ\psi_2.\label{eqn:outernorm2a}
\end{align}
and the constant terms on the pair of maps represented by the square-bracketed expression are again zero.
For a good approximate two-cycle $(G_1,F_1),(G_2,F_2)$ we therefore expect the corresponding constant terms to be small.
Optimising the locations $(c_k,d_k)$ numerically, for a good approximate two-cycle chosen with truncation degree $N=24$ yielded (displayed to $5$dp):
\begin{align*}
    (c_{1}, d_{1})
    &\simeq (0.64623, 0.55421), \qquad
    (c_{2}, d_{2})
    \simeq (0.74835, 0.63468).
\end{align*}
Taking these centres as fixed, we now turn our attention to the radii $(r_1,s_1),(r_2,s_2)$.

\subsection{Domain radii}\label{sec:domainradii}

Consider a general composition of functions $H\in\mathscr{A}(\Omega_j)$ and $G,F\in\mathscr{A}(\Omega_k)$, under the assumption that $G,F:\Omega_k\to\Omega_j$. Using our chosen representation for functions analytic on polydiscs, note that $H\circ(G,F)$ is represented by
\begin{equation}
\widehat{H}\circ\left[\psi_j\circ(\widehat{G},\widehat{F})\right]\circ\psi_k=:\widehat{K}\circ\psi_k=:K,
\label{eqn:hatcompose}
\end{equation}
where $\widehat{H},\widehat{G},\widehat{F}\in\mathscr{A}(\Omega_0)$. Recall the discussion in Section~\ref{sec:rigorousframework}: in order that the leftmost composition in~(\ref{eqn:hatcompose}) (resulting in $\widehat{K}$) also lies in $\mathscr{A}(\Omega_0)$, so that $K\in\mathscr{A}(\Omega_k)$, we require that the norms of the function pair represented by the square-bracketed expression are less than one (and, ideally, as small as possible).

Such a condition must hold at every function composition in the functional equations for the two-cycle. Below, we will derive expressions for some of the corresponding norms.
Systematic experimentation for the radii led to the choices
\begin{align}
(r_1,s_1)&\simeq(1.6, 1.38),\label{eqn:chosenradii12}\qquad
(r_2,s_2)\simeq(1.25, 1.95).
\end{align}

\subsubsection{Estimates for 
Composition Norms}

For inner composition norms (i.e., the norms associated with the inner scaling operations in the expression for the operators $T_k$) we have on $\Omega_1$:
\begin{align}
(G_2^2,F_2)\circ\mathrm{diag}(a_2^2,b_2)
&= (\widehat{G}_2^2,\widehat{F}_2)\circ\left[\psi_2\circ\mathrm{diag}(a_2^2,b_2)\circ\psi_1^{-1}\right]\circ\psi_1.
\label{eqn:scalecompose}
\end{align}
We note that the map in the square braces is given by
\begin{align*}
(X,Y)
&\mapsto\left(
\frac{(a_2^2c_1-c_2)+(a_2^2r_1)X}{r_2},
\ \frac{(b_2d_1-d_2)+(b_2s_1)Y}{s_2}
\right).
\end{align*}
Note that $(\widehat{G}_2^2,\widehat{F}_2)\in \mathscr{A}(\Omega_0)^2$. We require that the pair of maps denoted by the expression in square braces also lies in $\mathscr{A}(\Omega_0)^2$ so that the leftmost composition in the right-hand side of~(\ref{eqn:scalecompose}) is well-defined $\mathscr{A}(\Omega_0)^2$. Thus we require that
\begin{equation*}
\mathrm{max}\left(
\frac{1}{r_2}|a_2^2c_1-c_2|+\frac{r_1}{r_2}|a_2|^2,
\ \frac{1}{s_2}|b_2d_1-d_2|+\frac{s_1}{s_2}|b_2|
\right)<1.
\end{equation*}
For the radii chosen~(\ref{eqn:chosenradii12}), the above norms are approximately $0.66951$ and $0.66837$, respectively.
Similarly, considering the corresponding maps on $\Omega_2$
the relevant norms are approximately $0.48225$ and $0.64653$.

%

For the outer composition norms, numerical estimates for the norms of the pairs of maps in the square bracketed expressions in~(\ref{eqn:outernorm1a}) and~(\ref{eqn:outernorm2a}) are given by $(0.56121, 0.49948)$ and $(0.72106, 0.62096)$ respectively, for the good approximate two-cycle maps.
With the choices for centres and radii detailed above, the domain extension conditions are also satisfied for both bidiscs $\Omega_1$ and $\Omega_2$. The computer-assisted portion of the proof verifies this via rigorous enclosure.

\subsection{Rigorous verification of domain extension}

\begin{figure}
\centering
\begin{tabular}{c}
\includegraphics[width=0.9\textwidth]{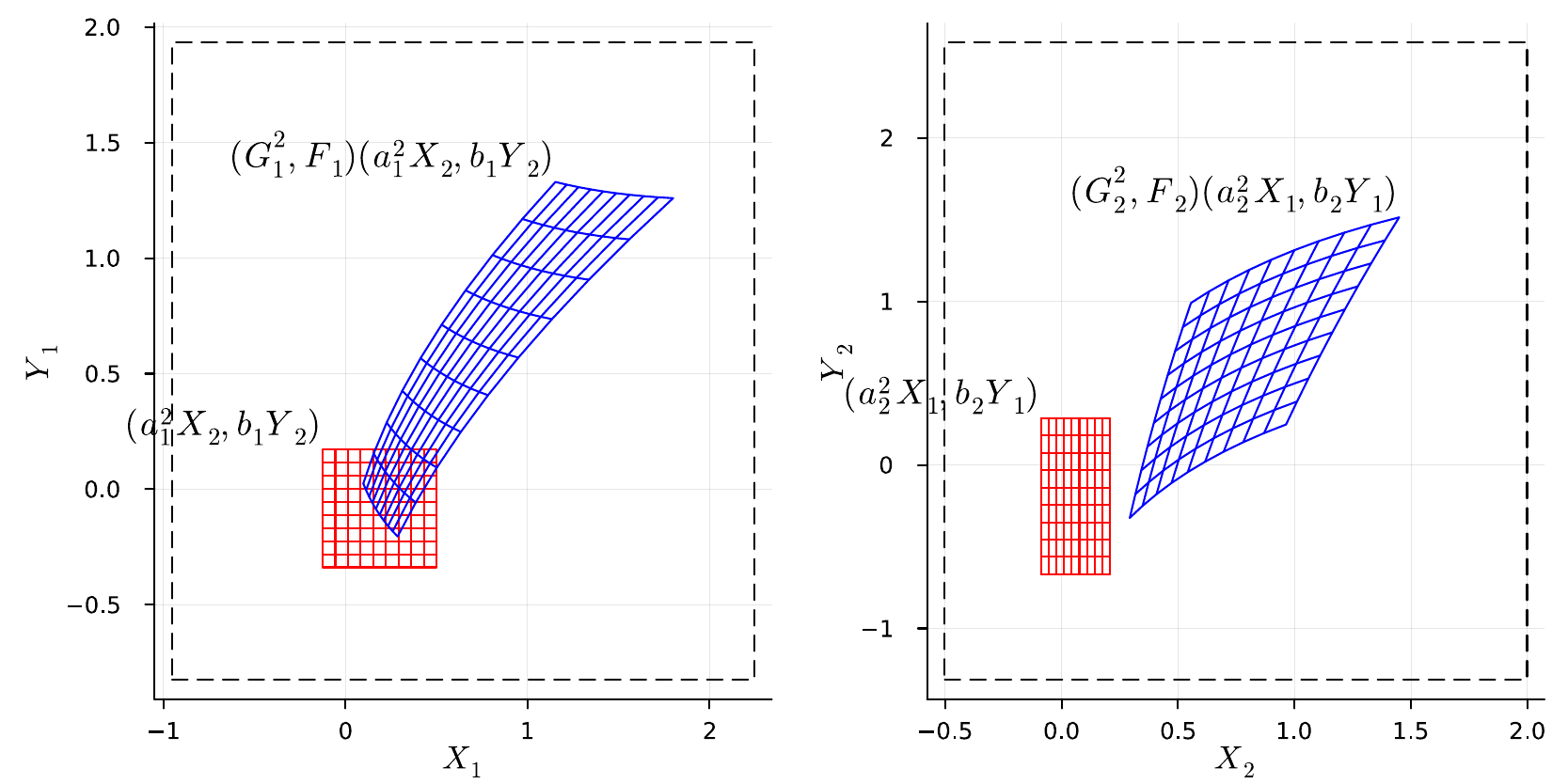}\\
Real domain extension holds for domains $\Xi_1$,$\Xi_2$.
\end{tabular}
\caption{Real domain extension. Shown are the images of the real domains under the relevant maps in~(\ref{eqn:tde1})--(\ref{eqn:tde4}).
Showing containments $\overline{\mathrm{diag}(a_1^2,b_1)\Xi_2}\subset\Xi_1$ (red) and $\overline{(G_1^2,F_1)\mathrm{diag}(a_1^2,b_1)\Xi_2}\subset\Xi_1$ (blue) on the left and
$\overline{\mathrm{diag}(a_2^2,b_2)\Xi_1}\subset\Xi_2$ (red) and $\overline{(G_2^2,F_2)\mathrm{diag}(a_2^2,b_2)\Xi_1}\subset\Xi_2$ (blue) on the right for
$\Xi_1:=\Re D(c_1,r_1)\!\times\! \Re D(d_1,s_1)\subset\mathbb{R}^2$ where
$c_1\simeq 0.64623$, $r_1=1.6$,
$d_1\simeq 0.55421$, $s_1=1.38$ and 
$\Xi_2:=\Re D(c_2,r_2)\!\times\! \Re D(d_2,s_2)\subset\mathbb{R}^2$ where
$c_2\simeq 0.74835$, $r_2=1.25$,
$d_2\simeq 0.63468$, $s_2=1.95$ (with boundaries $\partial\Xi_k$ shown as dashed lines).}
\label{fig:de_real}
\end{figure}

\begin{figure}
\centering
\begin{tabular}{c}
\includegraphics[width=0.9\textwidth]{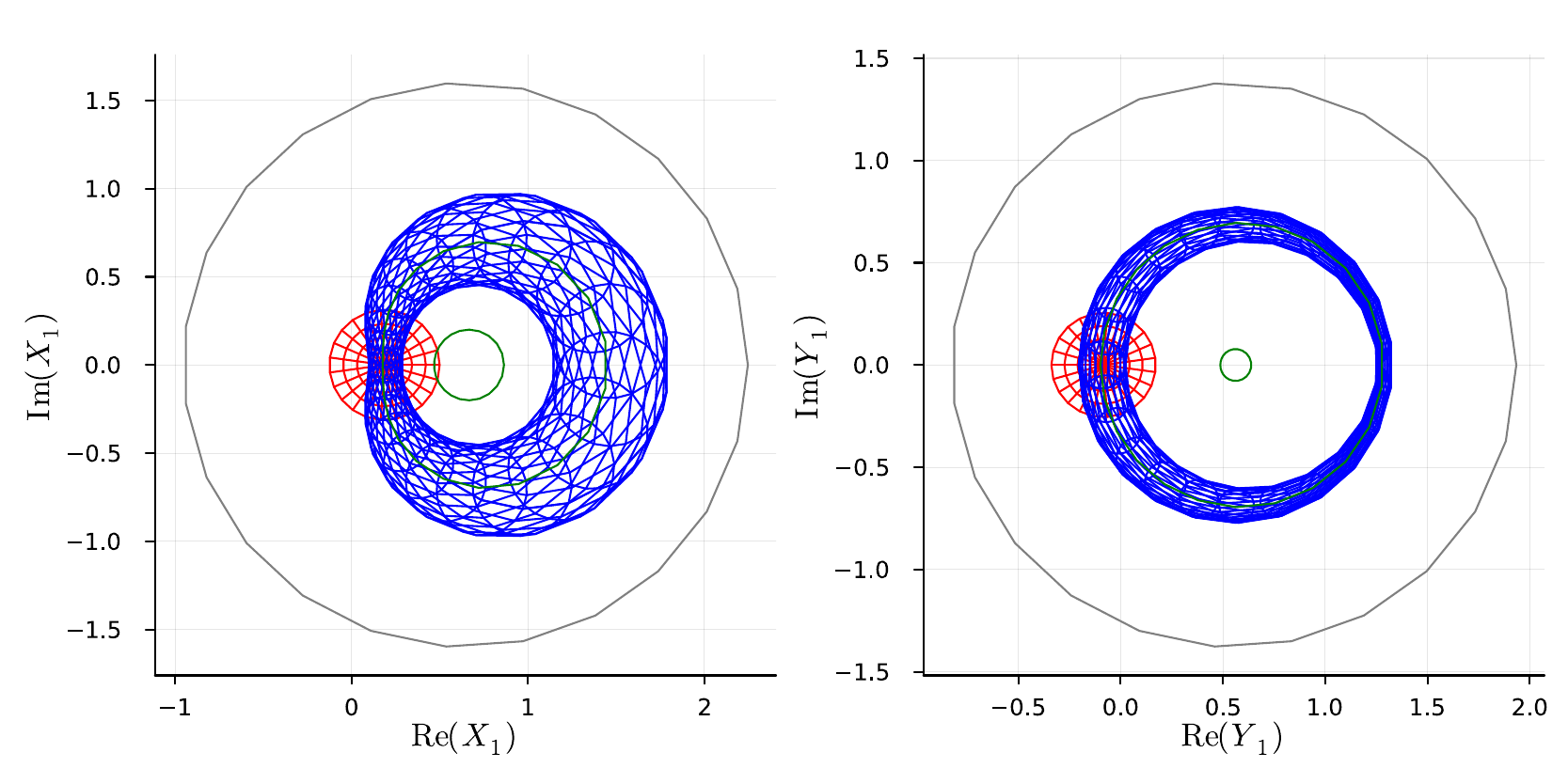}\\
(a) Containment in $\Omega_1$.\\
\includegraphics[width=0.9\textwidth]{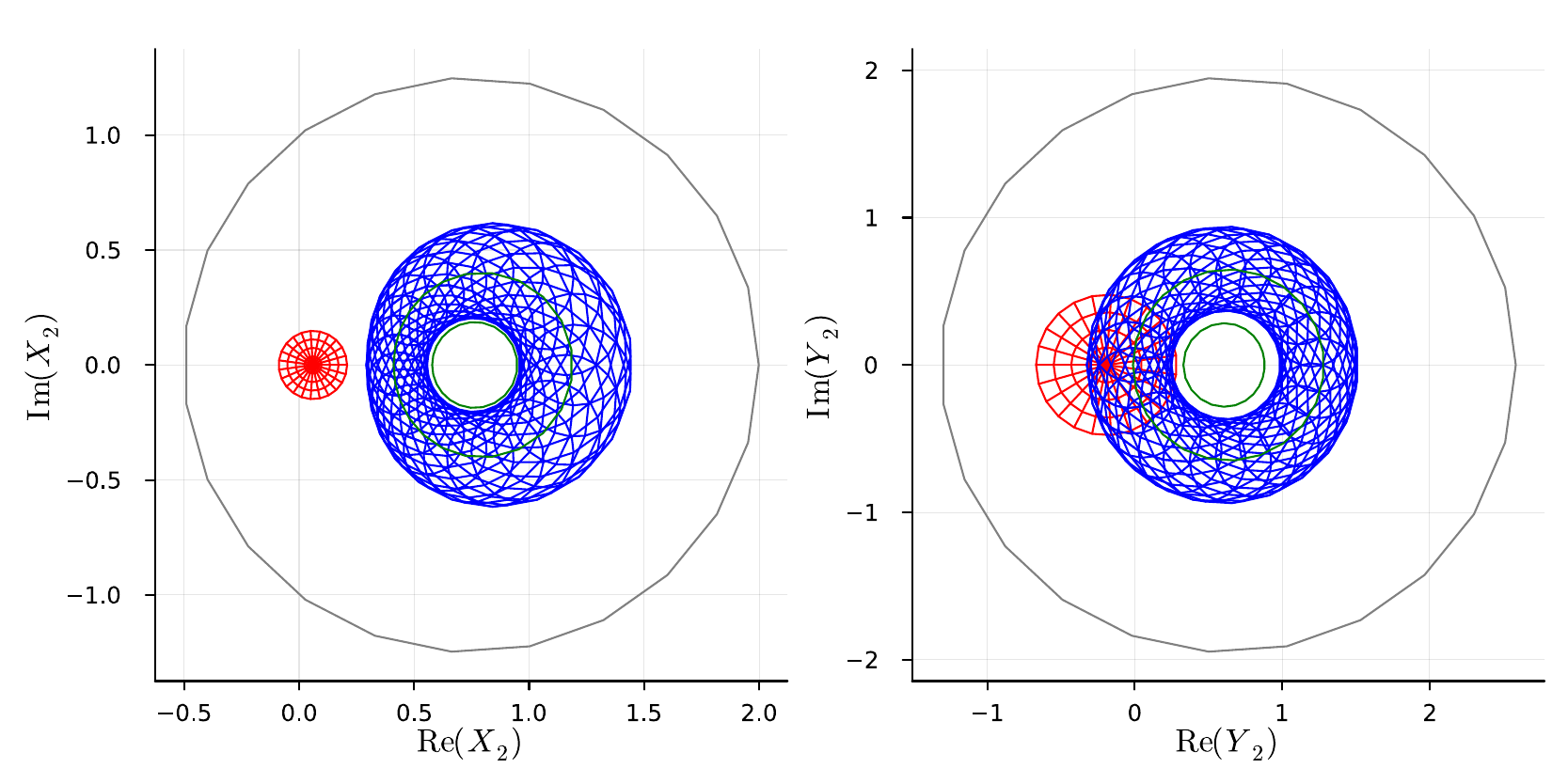}\\
(b) Containment in $\Omega_2$.
\end{tabular}
\caption{Complex domain extension. (a) Top:
$\Omega_1=D(c_1,r_1)\!\times\! D(d_1,s_1)$ (boundary shown as outer circles), where
$c_1\simeq 0.64623$, $r_1=1.6$,
$d_1\simeq 0.55421$, $s_1=1.38$. 
(b) Bottom:
$\Omega_2=D(c_2,r_2)\!\times\! D(d_2,s_2)$ (boundary shown as outer circles), where
$c_2\simeq 0.74835$, $r_2=1.25$,
$d_2\simeq 0.63468$, $s_2=1.95$.
(For illustration purposes only; the rigorous verification of domain extension uses a covering of the distinguished boundary by cartesian products of rectangles.) Shown are the images of the distinguished boundaries $\flat\Omega_k$ (tori) under the relevant maps in~(\ref{eqn:tde1})--(\ref{eqn:tde4}). Rescaled distinguished boundaries $\overline{\mathrm{diag}(a_k^2,b_k)\flat\Omega_j}$ are shown in red in colour version, with images $\overline{(G_k^2,F_k)\mathrm{diag}(a_k^2,b_k)\flat\Omega_j}$ of the distinguished boundaries shown in blue in colour version, and images of the degenerate tori (circles $\partial D(c_k,r_k)\times\{0\}$ and $\{0\}\times \partial D(d_k,s_k)$) shown in green. Compare (a) with Figure~\ref{fig:de_real}(left) and (b) with Figure~\ref{fig:de_real}(right) showing the restriction to real domains.
}
\label{fig:de}
\end{figure}

Let $\flat\Omega\cong\mathbb{T}^2$ denote the distinguished boundary of the bidisc. We note that for analytic functions $f\in\mathscr{A}(\Omega_k)$ (continuous on $\overline{\Omega_k}$), the maximum principle for several complex variables ensures that the maximum modulus is obtained on $\flat\Omega_k$, and the values of the function there determine the values of the function on the closure of the domain (the distinguished boundary is the Shilov boundary for the polydisc)~\cite{rudin_function_1969}. Thus
\begin{equation*}
\overline{f\left(\flat\Omega_k\right)}
\subset \Omega_j
\Rightarrow
\overline{f\left(\Omega_k\right)}
\subset \Omega_j,
\end{equation*}
and it is therefore enough to establish that the closure of the image of a covering of the distinguished boundary lies strictly within the domain. We use this strategy to visualise complex domain extension for good domains in Figure~\ref{fig:de} with the corresponding restrictions to real domains shown in Figure~\ref{fig:de_real}(b).
In the computer-assisted portion of the proof, to verify these conditions we cover each disc $D(c,r)$, $D(d,s)$ with a union of rectangles, and take the pairwise cartestian products of these sets to produce a rigorous covering of the (filled) bidisc $\Omega$. We then compute rigorous enclosures of the images of this covering, valid for all maps in ball of functions defined in Section~\ref{sec:results}.

\section{Quasi-Newton operator}\label{sec:quasinewton}

In this section, we will formulate a Newton-like operator, $\Phi$, whose fixed points are two-cycles of $T$, and later establish that it is a uniform contraction mapping on a carefully-chosen ball $B(V^0,\rho)$ around the good approximate two-cycle $V^0$.

Recall that $\mathcal{T}=\mbox{diag}(T_1,T_2)$
represents the action of $T$ on pairs $(P_1,P_2)$
and that 
in order to find a two-cycle, we seek roots of the corresponding operator $\mathcal{F}=I-\mathcal{S}\mathcal{T}$.
Ideally, we would like to establish that
the Newton operator for $\mathcal{F}$ on the space $A:=\mathscr{A}(\Omega_1)^2\oplus\mathscr{A}(\Omega_2)^2$ of pairs of pairs $V=(P_1,P_2)=((G_1,F_1),(G_2,F_2))$, viz.
\begin{equation*}
V \mapsto V - [\mathfrak{D}\mathcal{F}(V)]^{-1}[\mathcal{F}(V)],
\end{equation*}
is a contraction map on a suitably-chosen ball in the space. However, doing this directly would involve bounding the second Fr\'echet derivative of $\mathcal{F}$. Instead, we formulate a quasi-Newton operator $\Phi$ whose fixed-points are two-cycles of $T$ as follows.
Let $V^0$ denote the good approximate two-cycle found earlier.
For a fixed linear operator $\Gamma\simeq[I-S\Delta]^{-1}$ with $\Delta\simeq \mathfrak{D}\mathcal{T}(V^0)$, and $\Delta \mathcal{H}A=\{0_{A}\}$, the corresponding quasi-Newton operator is given by
\begin{equation}
\Phi: V \mapsto V - \Gamma [\mathcal{F}(V)],\label{eqn:quasinewtonop}
\end{equation}
with Fr\'echet derivative given by
\begin{equation}
\mathfrak{D}\Phi(V):\delta V\mapsto
\delta V - \Gamma[\mathfrak{D}\mathcal{F}(V)\delta V].\label{eqn:polynorms}
\end{equation}

\subsection{Contractivity}

We will show that the operator $\Phi$, is a uniform contraction mapping on a  ball $B(V^0,\rho)$ around $V^0$ by bounding the error on the approximate two-cycle $V^0$ and finding a uniform Lipschitz constant for $\Phi$:
\begin{align}
\left\|\Phi(V^0)-V^0\right\|
&\le \varepsilon,\label{eqn:epsilonbound}\\
\|\Phi(U)-\Phi(V)\|
&\le \kappa\|U-V\|,
\quad\mbox{for all $U,V\in B(V^0,\rho)$ for some $\kappa<1$},\label{eqn:kappabound}
\end{align}
It then suffices to verify the inequality $\varepsilon < \rho(1-\kappa)$ in order to show that the ball is mapped strictly inside itself.

In order to identify a uniform Lipschitz constant for~(\ref{eqn:kappabound}), we first note that the operator norm of $\mathfrak{D}\Phi(V)$ is bounded by the $(\mathrm{sum},\ell_1)$-norm  $\|\cdot\|_{\mathrm{sum},1}$ of the action of $\mathfrak{D}\Phi(V)$ on the (denumerable collection of) basis elements of $A$:
\begin{equation}
\|\mathfrak{D}\Phi(V)\| := \sup_{\|U\|=1}\|\mathfrak{D}\Phi(V)U\| \le \sup_{1\le i\le 4;\ j,k\ge 0}\|\mathfrak{D}\Phi(V)e_{ijk}\|.\label{eqn:columnsumnorm}
\end{equation}
By the mean value theorem, we have
\begin{equation*}
\left\Vert \Phi(U) - \Phi(V) \right\Vert \leq \sup_{W} \left\Vert \mathfrak{D}\Phi(W)\right\Vert \cdot \left\Vert U - V \right\Vert,
\end{equation*}
which follows by making a one-dimensional argument along line segments $[U,V]$ and using convexity of the ball $B(V^0,\rho)$.

\subsection{Implementing the Quasi-Newton Operator}

The computation of bounds on Lipschitz constants in~(\ref{eqn:columnsumnorm}) relies on the form chosen for the fixed linear operator $\Gamma$ in~(\ref{eqn:quasinewtonop}).
Partitioning the basis of $A$ into components $\#_1,\#_2$ corresponding to the pairs $P_1,P_2$, the Fr\'echet derivative of the operator $\mathcal{F}$ may be decomposed as
\begin{equation*}
\mathfrak{D}\mathcal{F}(P_1, P_2) = \left( \begin{array}{c|c}
I & -\mathfrak{D}T_2(P_2)  \\
\hline
-\mathfrak{D}T_1(P_1) & I
\end{array} \right)\begin{array}{c}\#_1\\ \#_2\end{array}.
\end{equation*}
Partitioned into polynomial and high-order parts, on each pair we approximate
\begin{equation*}
\mathfrak{D}T_k(P_k^0) \simeq \Delta_k
:= \left( \begin{array}{c|c}
\Delta_{k,PP} & 0  \\
\hline
0 & 0
\end{array} \right)\begin{array}{c}\mathscr{P}\\ \mathscr{H}\end{array}.
\end{equation*}
Permuting the basis elements of $(P_1,P_2)\in\mathscr{A}(\Omega_1)^2\oplus\mathscr{A}(\Omega_2)^2$, as follows,
\begin{align*}
\mathfrak{D}\mathcal{F}(P_1, P_2)
&\simeq
\left( \begin{array}{cc|cc}
I &-\Delta_{2,PP} & 0 & 0 \\
- \Delta_{1,PP} & I & 0 & 0 \\ 
\hline
0 & 0 & I & 0 \\
0 & 0 & 0 & I
\end{array} \right)
\begin{array}{c}\mathscr{P}\#_1\\ \mathscr{P}\#_2\\ \mathscr{H}\#_1\\ \mathscr{H}\#_2\end{array}
=: I-S\Delta,
\end{align*}
gives the latter in block-diagonal form, which enables us to write its inverse as
\begin{equation}
(I-S\Delta)^{-1}=:\Gamma
=
\left(\begin{array}{c|c}\Gamma_{PP}&0\\ \hline 0&I\end{array}\right)
\begin{array}{c}\mathscr{P}\\ \mathscr{H}\end{array}.\label{eqn:gammadecomposed}
\end{equation}
Note that the action of $\Gamma$ on the high-order part of the space is therefore the identity.
Let $V = (V_1,V_2,V_3,V_4):=(G_1, F_1, G_2, F_2)$ where each of the 4 maps are divided into their polynomial, higher order, and error parts; $V=V_P+V_H+V_E$ with $\|V_H\|\le v_H$ and $\|V_E\|\le v_E$. We will use the following decomposition,
\begin{equation}
\Gamma V = [\Gamma_{PP} V_P] + [V_H] + [\Gamma_{PP} \mathcal{P} V_E + \mathcal{H} V_E],\label{eqn:gammaparts}
\end{equation}
to define the polynomial, high-order, and error-parts (square-bracketed expressions in~(\ref{eqn:gammaparts})) of $\Gamma V$, respecively.
Denoting $W_E:=\Gamma_{PP}\mathcal{P}V_E+\mathcal{H}V_E$ note that, for example
\begin{align*}
\|(W_E)_1\|
&\le\max(1,\|\gamma_{11}\|)v_{1,E}
+\|\gamma_{12}\|v_{2,E}
+\|\gamma_{13}\|v_{3,E}
+\|\gamma_{14}\|v_{4,E},
\end{align*}
and similarly for $\lVert (W_E)_i\rVert$ for $i=2,3,4$, 
in which $\Gamma_{PP}$ is decomposed into component maps $\gamma_{mn}:\mathcal{P}V_n\mapsto(\Gamma_{PP}\mathcal{P}V_n)_m$ with maximum column-sum norms $\|\gamma_{pq}\|$.
Using this decomposition takes advantage of the fact that $V_E=\mathcal{P}V_E+\mathcal{H}V_E$ is a disjoint partition so that $\|V_E\|=\|\mathcal{P}V_E\|+\|\mathcal{H}V_E\|\le v_E$. Contrast this with the decomposition $\Gamma V = [\Gamma_{PP} V_P] + [V_H+\mathcal{H} V_E] + [\Gamma_{PP} \mathcal{P} V_E]$ in which we would be forced to bound the middle (high-order) term by $\|V_H+\mathcal{H}V_E\|\le v_H+v_E$.

\subsection{High-order perturbations}

We will split the computation of the norms on the right-hand side of~(\ref{eqn:columnsumnorm}) into polynomial versus high-order basis elements. For the latter, $\delta V\in\mathcal{H}A$, $\mathcal{P}\delta V=\{0\}$, and
\begin{align}
\mathfrak{D}\Phi(V)\delta V
&= \delta V - \Gamma \mathfrak{D}\mathcal{F}(V)\delta V\label{eqn:uncancelled}\\
&= \delta V - \delta V + \Gamma S\mathfrak{D}\mathcal{T}(V)\delta V\label{eqn:cancel}\\
&= \Gamma S\mathfrak{D}\mathcal{T}(V)\delta V,\label{eqn:highnorms}
\end{align}
which follows from~(\ref{eqn:gammadecomposed}).
Thus the expression~(\ref{eqn:polynorms}) may be used when bounding the norms $\mathfrak{D}\Phi(V)\delta V$ for polynomial basis elements, but expression~(\ref{eqn:highnorms}) must be used when bounding those for high-order basis elements in order to avoid the redundant terms $\delta V-\delta V$~(\ref{eqn:cancel}) implicit in~(\ref{eqn:uncancelled}). (Computing the enclosure of an expression such as $\delta G_k-\delta G_k$ without regard for dependencies between the operands, for $\delta G_k\in\mathscr{B}(\mathbf{0};1,0)$ where $\mathbf{0}$ denotes the constant zero function $\mathbf{0}:(X,Y)\mapsto 0$, would give
\begin{equation*}
\delta G_k-\delta G_k
\in\mathscr{B}(\mathbf{0};1,0)
-\mathscr{B}(\mathbf{0};1,0)
\subseteq\mathscr{B}(\mathbf{0};2,0),
\end{equation*}
which precludes the computation of an upper bound less than one.)
For the subexpression $\mathfrak{D}\mathcal{T}(V)$ appearing in~(\ref{eqn:highnorms}) and implicitly in~(\ref{eqn:polynorms}), we compute rigorous enclosures of the Fr\'echet derivative using the expressions given in~(\ref{eqn:frechett1})--(\ref{eqn:frechettn}).

\section{Results}\label{sec:results}

\subsection{Existence of the two cycle}

\begin{thm}[Existence of the C-type two-cycle]\label{thm:existence} The conjectured C-type two-cycle $V^*=((G_1^*,F_1^*),(G_2^*,F_2^*))$ exists. The state space scaling constants are bounded by
\begin{align*}
\alpha_1^*
:=1/a_1^*
&\in[-2.000377729,\ -2.000377707],\\
\beta_1^*
:=1/b_1^*
&\in[-7.646886669,\ -7.646886254],\\
\alpha_2^*
:=1/a_2^*
&\in[-3.282055430,\ -3.282055360],\\
\beta_2^*
:=1/b_2^*
&\in[-2.892710470,\ -2.892710346],
\end{align*}
%
%
giving, for the state space scalings associated with quadrupling the period,
\begin{align*}
\alpha^{*}:=\alpha_1^*\alpha_2^*
&\in[6.565350378,\ 6.565350586],\\
\beta^{*}:=\beta_1^*\beta_2^*
&\in[22.12022698,\ 22.12022913].
\end{align*}
%
%
\end{thm}

Proof (outline):
Taking $j+k\le {N=24}$, the domains defined in Section~\ref{sec:domaindetails}, and the good approximate two-cycle found in Section~\ref{sec:approxcycle1},
we compute the bound
\begin{equation*}
\|\Phi(V^0)-V^0\|_{\mathrm{sum},1}
\le{\varepsilon=1.5089\times10^{-10}},
\end{equation*}
with respect to the product space norm $\|(G_1,F_1,G_2,F_2)\|_{\mathrm{sum},1}$, using the singleton function ball
\begin{equation*}
\mathscr{B}(V^0;0,0)
=\mathscr{B}(G_1^0;0,0)
\oplus\mathscr{B}(F_1^0;0,0)
\oplus\mathscr{B}(G_2^0;0,0)
\oplus\mathscr{B}(F_2^0;0,0).
\end{equation*}
Then, on a ball of (pairs of pairs of) functions, 
\begin{align}
B(V^0,\rho,\|\cdot\|_{\mathrm{sum,1}})
&\subset B(V^0,\rho,\|\cdot\|_{\mathrm{max,1}})\nonumber\\
&= B(G_1,\rho,\|\cdot\|_{\Omega_1,1})
\times B(F_1,\rho,\|\cdot\|_{\Omega_1,1})\nonumber\\
&\quad{}\times B(G_2,\rho,\|\cdot\|_{\Omega_2,1})
\times B(F_2,\rho,\|\cdot\|_{\Omega_2,1})\label{eqn:max1ball}\\
&\subset\mathscr{A}(\Omega_1)^2\oplus\mathscr{A}(\Omega_2)^2,\nonumber
\end{align}
centered on the approximate two-cycle $V^0$, of radius {$\rho=7.5443\times10^{-10}>\varepsilon$}, we prove that
domain extension holds for the domains $\Omega_1,\Omega_2\subset\mathbb{C}^2$ given earlier. Note that the latter ball~(\ref{eqn:max1ball}) is enclosed using
\begin{equation*}
B(V^0,\rho,\|\cdot\|_{\mathrm{max}, 1})
=\mathscr{B}(G_1^0;0,\rho)
\oplus\mathscr{B}(F_1^0;0,\rho)
\oplus\mathscr{B}(G_2^0;0,\rho)
\oplus\mathscr{B}(F_2^0;0,\rho).
\end{equation*}

\begin{figure}
    \centering
    \includegraphics[width=0.65\textwidth]{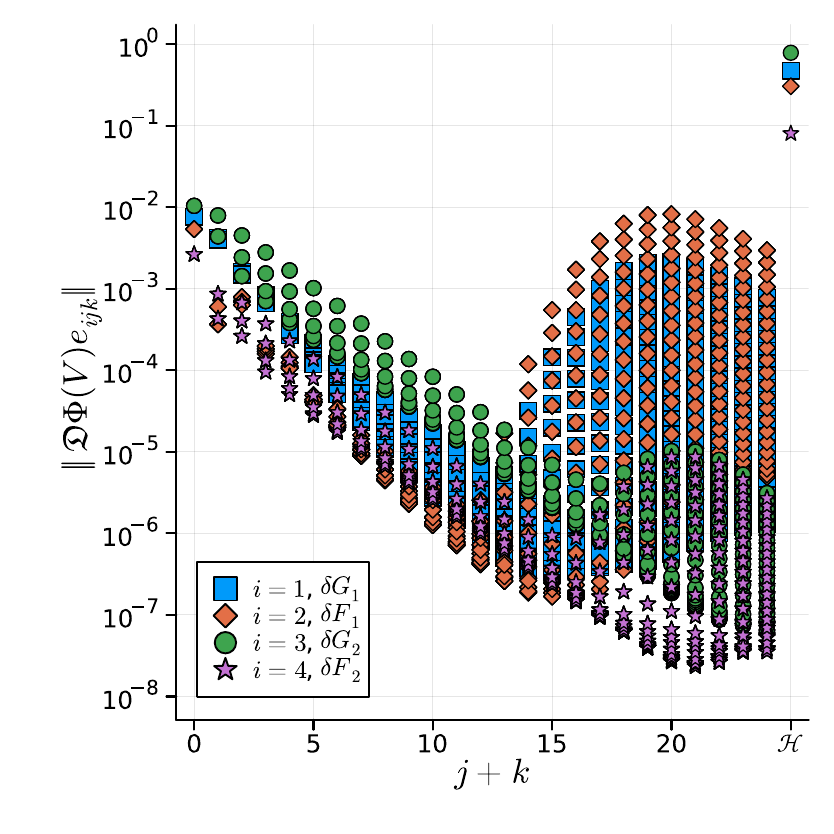}
    \caption{Upper bounds on the norms $\|\mathfrak{D}\Phi(V)e_{ijk}\|$ for $N_4=1300$ polynomial basis elements $e_{ijk}$, valid for all $V\in\mathscr{B}(V^0;0,\rho)$, plotted against the homogeneous degree $j+k\le N$ of the basis element (a log-scale is used for the vertical axis). The upper bounds on high-order norms (valid for all $e_{ijk}$ with $i=1,2,3,4$ and $j+k>N$) are labelled by $\mathcal{H}$ at the right-hand edge of the plot.}
    \label{fig:polynorms}
\end{figure}

Further, we prove that
the polynomial basis element norms of the form $\|\mathfrak{D}\Phi(V)\delta V\|$ are bounded by $0.0082192$ (Figure~\ref{fig:polynorms} plots upper bounds on individual norms against the homogeneous degree of the corresponding basis element). This bound is valid for all $V\in\mathscr{B}(V^0;0,\rho)$ and is computed using singleton function balls $\mathscr{B}(e_{ijk};0,0)\ni\delta V:=(\delta G_1,\delta F_1,\delta G_2,\delta F_2)$,
in which the $e_{ijk}$ are the basis elements for the product space $A:=\mathscr{A}(\Omega_1)^2\oplus\mathscr{A}(\Omega_2)^2$, e.g., $e_{3jk}=(\mathbf{0}, \mathbf{0}, e_{jk}, \mathbf{0})$, where $e_{jk}\in\mathscr{A}(\Omega_2)$.
Thus for $j+k\le N$, the corresponding polynomial perturbations $\delta G_\ell,\delta F_\ell\in\mathcal{P}\mathscr{A}(\Omega_\ell)$, written in the product basis, are enclosed using, for the first component map $i=1$,
\begin{align*}
\mathscr{B}(e_{jk};0,0)\oplus\{\mathbf{0}\}
\oplus\{\mathbf{0}\}\oplus\{\mathbf{0}\}
&\ni\delta V
=(\delta G_1, \mathbf{0},\mathbf{0},\mathbf{0}),
\end{align*}
and analogously for $i=2,3,4$. For $N=24$ there are $N_4=1300$ such basis elements.
We compute the corresponding norms for high-order perturbations using $4$ nondegenerate function balls (of high-order radius $1$), one for each component map:
\begin{align*}
\mathscr{B}(\mathbf{0};1,0)\oplus\{\mathbf{0}\}
\oplus\{\mathbf{0}\}\oplus\{\mathbf{0}\}
&\ni\delta V
=(\delta G_1, \mathbf{0},\mathbf{0},\mathbf{0}),
\end{align*}
for $i=1$ (similarly for $i=2,3,4$)
where $\delta G_\ell,\delta F_\ell\in\mathcal{H}\mathscr{A}(\Omega_\ell)$.
Note that these are high-order projections of the unit ball; each is the convex hull of all high-order basis elements for one of the four component maps (this is in contrast to the function balls containing the polynomial basis elements, all of which are singletons). The four corresponding upper bounds produced are
$0.47570$, $0.30583$, $0.78693$, and $0.08044$, respectively. We therefore have the upper bound
\begin{equation*}
\|\mathfrak{D}\Phi(V)\|\le{\kappa=0.78693}<1,
\end{equation*}
which gives a uniform Lipschitz constant for $\Phi$, valid for all $V\in\mathscr{B}(V^0;0,\rho)$.
Since {$\varepsilon<\rho(1-\kappa)$} then, by the Contraction Mapping Theorem,  a locally-unique two-cycle $V^*=(G_1^*,F_1^*,G_2^*,F_2^*)\in\mathscr{B}(V^0;0,\rho)$ of $T$ (and hence that of $R$) exists.\qed

\subsubsection*{Nondegeneracy}

In order to prove that the two-cycle is non-degenerate, i.e., that we do not have $(G_1^*,F_1^*)=(G_2^*,F_2^*)$, in the sense that both pairs represent the same analytic functions, we note that different power series representing the same analytic function on overlapping domains must produce the same function value at all points in their intersection. It is therefore enough to find a single point $(Z,W)\in\Omega_1\cap\Omega_2$
for which the intervals enclosing the function values $G_1(Z,W)$ and $G_2(Z,W)$, valid for all function pairs $(G_1,G_2)\in\mathscr{B}(G_1^0;0,\rho)\oplus\mathscr{B}(G_2^0;0,\rho)$, 
are disjoint (similarly for $F_1(Z,W)$ and $F_2(Z,W)$). In the computer-assisted portion of the proof, we do this for the point $(Z,W)=(1/2,1/2)$, for which
\begin{align*}
G_1(1/2,1/2)\in[0.18062, 0.18063]
&<[0.34655, 0.34656]\ni G_2(1/2,1/2),\\
F_1(1/2,1/2)\in[0.20474, 0.20475]
&<[0.56211, 0.56212]\ni F_2(1/2,1/2).
\end{align*}

\section{Computational considerations}\label{sec:computational}

The computer-assisted portion of the proof was implemented in the \texttt{Julia} Programming Language 1.10.1~\cite{bezanson2017julia}. Previous work including~\cite{burbanks_rigorous_2021a,burbanks_existence_2023} (with supporting software~\cite{burbanks_2021_code_5608449,burbanks_2022_code_7139006}) utilised multiprecision arithmetic. However, even with the relatively low truncation degree $N=24$ used in this paper, working with maps of two variables leads to prohibitive computation times. The interval and function ball arithmetic are therefore implemented using standard 64-bit floating point numbers with IEEE-754 compliant accuracy for basic operations. In order to guarantee rigorous enclosures for basic operations, independent verifications were done using fast rounding of bounds on elementary operations outward by one ULP to the nearest floats and separately via \texttt{BigFloat}~\cite{bezanson2017julia} with directed rounding modes. Parallel computation was used for matrix elements, and for the computation of the Lipschitz constants by bounding the action of the Fr\'echet derivative on  polynomial and high-order basis elements. Multiprocessing was chosen for this  (via the \texttt{Julia} Distributed package) in order to respect the safety of the directed rounding modes at the per-process level. Closures were used in the implementation of expressions for the Fr\'echet derivative in order to avoid recomputation of common subexpressions. Choosing good domains is crucial for the proof to succeed with low truncation degree $N$ for the polynomial parts of the function spaces; the computation of the bounds on $\|\mathfrak{D}\Phi(V)\delta V\|$ provides the biggest challenge determined by the size $N_4+4=2(N+1)(N+2)+4$ of the rigorous enclosure of the basis. The truncation degree must be chosen sufficiently high in order that the corresponding norms are less than $1$ for high-order perturbations $\delta V$. Binary arithmetic was used in the computer-assisted portion of the proof; all bounds quoted in this paper have been safely rounded to nearby decimals. To ensure that the proof can be reproduced, the code and data files are provided in a Zenodo repository~\cite{burbanks_2026_ctype_code}.

\section{Conclusions and future work}\label{sec:conclusions}

We have proved the existence of the RG two-cycle of C-type conjectured by Kuznetsov et al.~\cite{kuznetsov_perioddoubling_1997}.
This extends previous work (on the FS-type RG fixed point~\cite{burbanks_existence_2023}) for non-invertible maps of the plane from unidirectionally-coupled to fully-coupled pairs of maps. It also extends the framework from existence proofs for RG fixed points to proofs for periodic cycles.
Together with an upcoming publication, in which we prove the existence of the so-called FQ-type RG fixed point, this work helps to complete the picture for two-dimensional non-invertible maps given in~\cite{kuznetsov_perioddoubling_1997}.

We note that hyperbolicity of the C-type two-cycle remains an open conjecture: we would need to show that the relevant parts of the spectra $\sigma(\mathfrak{D}(R^2)(g_k^*,f_k^*))$ are contained strictly in the open unit disc, except for $2$ relevant expanding eigenvalues
(each of unit multiplicity) indicated earlier.
Numerically, we find that the leading elements of the spectrum of $\mathfrak{D}(T^2)(G_k^*,F_k^*)$ are $\delta_1\approx 92.431$, $\delta_2\approx 4.1924$, with an eigenvalue $\delta_3=\beta^*\approx 22.120$ that corresponds to a coordinate change, with the largest non-expanding eigenvalue given by $\approx 0.93326$. These values are compatible with those computed by Kuznetsov and Sataev~\cite{kuznetsov_perioddoubling_1997,kuznetsov_birth_2008}. Note that the spectrum of $\mathfrak{D}(R^2)(g_k^*,f_k^*)$ as formulated in~\cite{kuznetsov_perioddoubling_1997,kuznetsov_birth_2008} additionally has irrelevant expanding eigenvalues $\delta_4=\alpha^*\approx 6.5653$ and $\delta_5=\beta^*/\alpha^*\approx 3.3692$ that correspond to infinitesimal coordinate changes; these do not appear for $\mathfrak{D}(T^2)(G_k^*,F_k^*)$ as formulated in this paper due to the symmetry imposed by the ansatz and the form chosen for the normalisation conditions on $R$ and $T$.

This work provides a further step in proving the conjectured picture that a family of FS-type RG fixed points gives rise to a family of C-type RG two-cycles by a period-doubling in the dynamics of the renormalisation operator $R$ itself~\cite{kuznetsov_birth_2008}. Numerical evidence for a bifurcation from the FS-type fixed point for maps with critical point of degree $d=2$ begins with the spectrum of the derivative of $R$ there, which has a weakly expanding relevant eigenvalue $\lambda\simeq-1.018$ (close to $-1$). Thorough numerical investigations led to the conjecture that $d$ itself is a bifurcation parameter in an extended space of functions, and that $R$ undergoes period-doubling at a critical value $d=d_c<2$ resulting in a family of C-type two-cycles that intersects $d=2$ at the two-cycle studied in this paper. Numerically, it appears that the codimension $4$ FS-type fixed point has a relevant eigenvalue whose square collides with the irrelevant eigenvalue $\delta_3$ for the codimension $2$ C-type two-cycle at the RG bifurcation point. Tackling this question, one would begin by proving persistence of the FS-type fixed point and C-type two-cycle with perturbations of $d$ by extending the space (or equivalently by adjusting the operator) to incorporate $d$ as a parameter and using the implicit function theorem.

Recall that the C-type scaling can be observed in Rossler systems under external driving~\cite{kuznetsov_universality_2001} and more generally in systems
where variation of one parameter gives rise to period-doublings and another to saddle node bifurcations. A striking application of the C-type criticality is in certain biological models including those for nephron autoregulation by
Laugesen et al. 2011~\cite{laugesen_modelling_2011}. It would be useful to identify other examples of dynamical systems that display the C-type criticality.

Recent work by Breden, Gonzalez, and Mireles James~\cite{breden_validated_2024} extends computer-assisted proofs for renormalisation operators by using Chebyshev rather than Taylor series,  together with bounding the Discrete Fourier Transform. The work applies bounds on discretisation errors, alongside those on truncation errors, to the Cvitanovi\'c-Feigenbaum RG operator and its generalisations from period doubling cascades to scaling within collections of period $c\cdot M^n$ orbits for fixed $c$ and fixed $M\ge 2$ (where each collection corresponds to the iteration of a particular transformation on kneading sequences) occurring in families of maps with a degree $d\ge 2$ critical point. This results in a framework better-suited to real analytic maps. An extension of the framework to multiple variables would enable a more efficient version of the proof shown here.

A variety of other universal scaling scenarios have been conjectured (supported by strong numerical evidence) by Kuznetsov, Kuznetsov, and Sataev ~\cite{kuznetzov_new_1992,kuznetsov_multiparameter_2005,kuznetsov_variety_1997,kuznetsov_variety_1993}. These provide additional avenues for making computer-assisted proofs in this area.

\section*{Acknowledgements}

For the purpose of open access, the authors have applied a Creative Commons Attribution (CC-BY) licence to any Author Accepted Manuscript version arising from this submission.

\section*{Declarations}

\subsection*{Funding}

Andrew Burbanks would like to thank the Isaac Newton Institute
for Mathematical Sciences (INI), Cambridge, for support and
hospitality during an INI Retreat. This work was supported by EPSRC grant EP/Z000580/1.

The authors would also like to thank the International Centre for Mathematical Sciences (ICMS), Edinburgh, for support via Workshop Grant~2416-BUR, the London Mathematical Society (LMS) for support via the Scheme~1 Conference Grant~12502, and the Glasgow Mathematical Journal Trust (GMJT) for support via the GMJT ICMS Block Grant 2026/2027.

Zainab Rahman was supported by a PhD bursary from the School of Mathematics and Physics at the University of Portsmouth, UK,
through the REF2021 Quality-related Research (QR) funding allocation.

\subsection*{Data availability statement}

The data and code that support the findings of this study are openly available at the following Zenodo repository~\cite{burbanks_2026_ctype_code}.

\subsection*{Conflict of interest}

The authors have no conflicts of interest to disclose.

\section*{ORCID iDs}

\begin{itemize}

\item Andrew Burbanks \href{https://orcid.org/0000-0003-0685-6670}{https://orcid.org/0000-0003-0685-6670}

\item Zainab Rahman \href{https://orcid.org/0009-0000-6028-8358}{https://orcid.org/0009-0000-6028-8358}

\item Maria Pickett \href{https://orcid.org/0000-0002-3601-1989}{https://orcid.org/0000-0002-3601-1989}

\end{itemize}

\bibliographystyle{plainurl}
\bibliography{references}

\end{document}